\documentclass[ reqno,12pt,fleqn]{amsart}
\usepackage{amsmath, amsthm, amscd, amsfonts, amssymb, graphicx, color}
\usepackage[bookmarksnumbered, colorlinks, plainpages]{hyperref}

\newtheorem{theorem}{Theorem}[section]
\newtheorem{lemma}[theorem]{Lemma}

\newtheorem{corollary}[theorem]{Corollary}
\theoremstyle{definition}
\newtheorem{definition}[theorem]{Definition}

\theoremstyle{remark}
\newtheorem{remark}[theorem]{Remark}
\numberwithin{equation}{section}

\begin{document}
\title[Approximation in weighted Orlicz spaces]{Jackson-Stechkin type inequalities
for differentiable functions in weighted Orlicz spaces}
\author[R. Akg\"{u}n]{Ramazan Akg\"{u}n$^{\star}$}
\address{ BALIKES\.{I}R UNIVERSITY, FACULTY OF ARTS AND SCIENCES, DEPARTMENT
OF MATHEMATICS, \c{C}A\u{G}I\c{S} YERLE\c{S}KES\.{I}, 10145, BALIKES\.{I}R, T%
\"{U}RK\.{I}YE.}
\email{\textcolor[rgb]{0.00,0.00,0.84}{rakgun@balikesir.edu.tr}}
\keywords{Jackson inequality, Moduli of Smoothness, Muckenhoupt weight,
Trigonometric Approximation.\\
$^{\star}$This research was supported by Balikesir University Research
Project 2019/061}

\begin{abstract}
In the present work some Jackson Stechkin type direct theorems of
trigonometric approximation are proved in Orlicz spaces with weights
satisfying some Muckenhoupt's $A_{p}$ condition. To obtain refined version
of the Jackson type inequality we prove an extrapolation theorem,
Marcinkiewicz multiplier theorem and Littlewood Paley type results. As a
consequence refined inverse Marchaud type inequalities are obtained. By
means of a realization result we find an equivalence between the fractional
order weighted modulus of smoothness and the classical weighted Peetre's \ $%
K $-functional.
\end{abstract}

\maketitle

\section{Introduction}

The present work is devoted to obtain some Jackson type direct theorems of
trigonometric approximation in weighted Orlicz spaces $L_{W}^{\varphi }$
that constructed by means of a quasiconvex Young function $\varphi $. The
Muckenhoupt's weights $W\in A_{p}$ ($1\leq p\leq \infty $) play a special
role in Harmonic Analysis because these are precisely the weights for which
some singular integrals and maximal operators are bounded in the weighted
Lebesgue spaces $L_{W}^{p}$ (\cite{c85}). We refer, for example, to the
monograph \cite{c85} for a complete account on the theory of Muckenhoupt
weights. In this study we will focus on those Muckenhoupt weights that in $%
A_{p\left( \varphi \right) }$ where $p\left( \varphi \right) $ is the indice
(\cite{gk94})\ of $\varphi $: 
\begin{equation}
\frac{1}{p\left( \varphi \right) }:=\inf \left\{ p:p>0,\text{ }\varphi ^{p}%
\text{ is quasiconvex}\right\} .  \label{indic}
\end{equation}%
A function $\varphi $ is called \textit{Young function} if $\varphi $ is
even, continuous function, nonnegative in $\mathbb{R}:=\left( -\infty
,+\infty \right) $, increasing on $\mathbb{R}^{+}:=\left( 0,\infty \right) $
such that%
\begin{equation*}
\varphi \left( 0\right) =0\text{,\quad }\underset{x\rightarrow \infty }{\lim 
}\varphi \left( x\right) =\infty \text{.}
\end{equation*}%
A function $\varphi :[0,\infty )\mathbb{\rightarrow \lbrack }0,\infty )$ is
said to be \textit{quasiconvex} ($QC$ briefly) if there exist a convex Young
function $\Phi $ and a constant $C\geq 1$ such that%
\begin{equation*}
\Phi \left( x\right) \leq \varphi \left( x\right) \leq \Phi \left( Cx\right) 
\text{,\quad }\forall x\geq 0\text{.}
\end{equation*}%
Detailed information of quasiconvex functions and their applications see
monographs \cite{GGKK,KK91}.

The Jackson type theorem relates the best approximation error $E_{n}\left(
f\right) _{\varphi ,W}$ from above with the modulus of smoothness. There is
a lot of way of defining the modulus of smoothness. Generally speaking the
modulus of smoothness are defined by means of the operations such as the
usual translation operator $f\left( \cdot \right) \rightarrow f\left( \cdot
+h\right) $, ($h\in \mathbb{R}$), or the average operator $f\left( \cdot
\right) \rightarrow \frac{1}{h}\int\nolimits_{0}^{h}f\left( \cdot +t\right)
dt$ and so on. Usual translation operator $f\left( \cdot \right) \rightarrow
f\left( \cdot +h\right) $ is, in general, not bounded for weighted spaces of
functions. For overcoming this difficulty we will consider the Steklov
average operator%
\begin{equation*}
f\left( \cdot \right) \rightarrow \mathcal{A}_{h}f\left( \cdot \right) :=%
\frac{1}{h}\int\nolimits_{-h/2}^{h/2}f\left( \cdot +t\right) dt
\end{equation*}%
and the Muckenhoupt weights. Considering weighted Lebesgue space $L_{W}^{p}$
with Muckenhoupt weight $W\in A_{p}$, ($1<p<\infty $), in 1986, E. A.
Gadjieva (\cite{Gadj}) defined the modulus of smoothness which is
constructed by means of the operator%
\begin{equation}
\sigma _{h}f\left( \cdot \right) :=\frac{1}{2h}\int\nolimits_{-h}^{h}f\left(
\cdot +t\right) dt,\text{\quad }h\in \mathbb{R}^{+}.  \label{st}
\end{equation}%
For the same spaces $L_{W}^{p}$, $W\in A_{p}$, ($1<p<\infty $), another
approach is given by N. X. Ky, see e.g., \cite{Ky1,Ky2,Ky3}. The Steklov
mean satisfies the inequality $\left\vert \mathcal{A}_{h}f\left( x\right)
\right\vert \leq Mf(x)$ almost everywhere on $\mathsf{T}:=[0,2\pi ]$, where $%
M$ is the Hardy-Littlewood maximal function. Considering the boundedness of
the Hardy-Littlewood maximal function $M$, author and Israfilov (\cite{ai11}%
) was considered the following modulus of smoothness in weighted Orlicz
spaces%
\begin{equation*}
\underset{\underset{i=1,2,...,r}{0<h_{i}\leq \delta }}{\sup }\left\Vert
\prod\limits_{i=1}^{r}\left( I-\sigma _{_{h_{i}}}\right) f\right\Vert
_{\varphi ,W}\text{,\quad }r\in \mathbb{N},
\end{equation*}%
where $\left\Vert \cdot \right\Vert _{\varphi ,W}$ is some norm in weighted
Orlicz space.

In this work we will consider more natural weighted modulus of smoothness%
\begin{equation*}
\Omega _{r}\left( f,\delta \right) _{\varphi ,W}:=\underset{0<h\leq \delta }{%
\sup }\left\Vert \left( I-\mathcal{A}_{_{h}}\right) ^{r}f\right\Vert
_{\varphi ,W}\text{,\quad }r\in \mathbb{R}^{+}
\end{equation*}%
of fractional order $r\in \mathbb{R}^{+}$. We obtain in \textsc{Theorem} \ref%
{teo2} that if%
\begin{equation}
k\in \mathbb{R}^{+}\text{, }\varphi \in \Delta _{2}\text{, }\varphi ^{\theta
}\text{ is }QC\text{ for some }\theta \in \left( 0,1\right) \text{, }W\in
A_{p\left( \varphi \right) }\text{, }f\in L_{W}^{\varphi }\text{,}
\label{kosul}
\end{equation}%
then%
\begin{equation}
E_{n}\left( f\right) _{\varphi ,W}\leq C_{k,W,\varphi }\Omega _{k}\left( f,%
\frac{1}{n}\right) _{\varphi ,W}  \label{dirr}
\end{equation}%
holds for $n\in \mathbb{N}$ with some constant $C_{k,W,\varphi }>0$
depending only on $k$, $\varphi $ and $W,$ where%
\begin{equation*}
E_{n}\left( f\right) _{\varphi ,W}:=\inf \left\{ \left\Vert f-T\right\Vert
_{\varphi ,W}:T\in \mathcal{T}_{n}\right\} \text{,}
\end{equation*}%
and $\mathcal{T}_{n}$ is the class of real trigonometric polynomials of
degree not greater than $n$.

Under the conditions (\ref{kosul}) we have also a refinement of the
inequality (\ref{dirr}). It is obtained in \textsc{Theorem} \ref{teo4}: If
the conditions (\ref{kosul}) are hold, then we have the inequality%
\begin{equation*}
\left( \prod\limits_{j=1}^{n}E_{j}\left( f\right) _{\varphi ,W}\right)
^{1/n}\leq C_{k,W,\varphi }\Omega _{k}\left( f,\frac{1}{n}\right) _{\varphi
,W}
\end{equation*}%
with a constant $C_{k,W,\varphi }>0$ depending only on $k$, $\varphi $ and $%
W $.

Let $\Psi $ be the class of increasing functions $\phi :[0,\infty
)\rightarrow \lbrack 0,\infty )$ satisfying $\phi \left( \infty \right)
:=\lim\limits_{x\rightarrow \infty }\phi \left( x\right) =\infty $ and let $%
0<p,q<\infty $. By $Y\left[ p,q\right] $ we denote the class of even
functions $\varphi \in \Psi $ satisfying \textbf{(i)} $\varphi \left(
u\right) /u^{p}$ is non-decreasing as $\left\vert u\right\vert $ increases; 
\textbf{(ii)} $\varphi \left( u\right) /u^{q}$ is non-increasing as $%
\left\vert u\right\vert $ increases.

The other refinement of (\ref{dirr}) is given in \textsc{Theorem} \ref{proc}%
: Let $\varphi \in Y\left[ p,q\right] $ for some $1<p,q<\infty $ and $%
\varphi \in \Delta _{2}$, $\varphi ^{\theta }$ is quasiconvex for some $%
\theta \in \left( 0,1\right) $. If $\beta :=\max \left\{ 2,q\right\} $ and $%
k\in \mathbb{R}^{+}$, $W\in A_{p\left( \varphi \right) }$, $f\in
L_{W}^{\varphi }$, then there exist positive constant $C_{k,W,\varphi }$%
\emph{\ }depending only on $k$, $\varphi $ and $W$ such that the refined
Jackson inequality%
\begin{equation*}
\frac{1}{n^{2k}}\left\{ \sum\limits_{\nu =1}^{n}\nu ^{2\beta k-1}E_{\nu
-1}^{\beta }\left( f\right) _{\varphi ,W}\right\} ^{1/\beta }\leq
C_{k,W,\varphi }\Omega _{k}\left( f,\frac{1}{n}\right) _{\varphi ,W}
\end{equation*}%
holds for any $n\in \mathbb{N}$.

To obtain the last inequality we require the following three results. The
first one is the Extrapolation \textsc{Theorem} \ref{ET}: Let $\varphi \in
\Delta _{2}$, $\varphi $ is quasiconvex, $\varphi \left( t^{1/p_{0}}\right) $
be a convex function for some $p_{0}\in \left( 1,\infty \right) $. Let $%
\mathcal{F}$ be a family of couples of nonnegative functions such that%
\begin{equation*}
\int\limits_{\mathsf{T}}F\left( x\right) ^{p_{0}}W\left( x\right) dx\leq
C\int\limits_{\mathsf{T}}g\left( x\right) ^{p_{0}}W\left( x\right) dx\text{%
,\quad }\left( F,g\right) \in \mathcal{F}
\end{equation*}%
for all $W\in A_{1}$ provided the left hand side is finite. Then, for all $%
W\in A_{p\left( \varphi \right) }$, the inequality%
\begin{equation*}
\int\limits_{\mathsf{T}}\varphi \left( F\left( x\right) \right) W\left(
x\right) dx\leq C\int\limits_{\mathsf{T}}\varphi \left( g\left( x\right)
\right) W\left( x\right) dx
\end{equation*}%
holds for any $\left( F,g\right) \in \mathcal{F}$ when the left-hand side is
finite.

The second is the Marcinkiewicz multiplier \textsc{Theorem} \ref{mmt}: Let a
sequence $\left\{ \lambda _{l}\right\} $\ of real numbers be satisfy the
properties%
\begin{equation*}
\left\vert \lambda _{l}\right\vert \leq A\text{,\quad }\sum%
\limits_{l=2^{m-1}}^{2^{m}-1}\left\vert \lambda _{l}-\lambda
_{l+1}\right\vert \leq A
\end{equation*}%
for all $l,m\in \mathbb{N}$, where $A$\ doesn't depend on $l$\ and $m$. If $%
\varphi \in \Delta _{2}$, $\varphi ^{\theta }$ is quasiconvex for some $%
\theta \in \left( 0,1\right) $, then there is a function $G\in
L_{W}^{\varphi }$\emph{\ }such that the series $\sum\nolimits_{k=0}^{\infty
}\lambda _{k}A_{k}\left( \cdot ,f\right) $\ is Fourier series of $G$\ and%
\begin{equation*}
\left\Vert G\right\Vert _{\varphi ,W}\leq C\left\Vert f\right\Vert _{\varphi
,W}
\end{equation*}%
holds with a positive constant $C$\ does not depend on $f$.

The third one is the Littlewood Paley type \textsc{Theorem} \ref{lpt}: If $%
\varphi \in \Delta _{2}$ is quasiconvex, $\varphi \left( t^{1/p_{0}}\right) $
is a convex function for some $p_{0}\in \left( 1,\infty \right) $, then
there are positive constants $c_{\varphi ,W},C_{\varphi ,W}$ depending only
on $\varphi $ and $W$ such that%
\begin{equation*}
c_{\varphi ,W}\left\Vert \left( \sum\limits_{l=0}^{\infty }\left\vert
\triangledown _{l}\right\vert ^{2}\right) ^{1/2}\right\Vert _{\varphi
,W}\leq \left\Vert f\right\Vert _{\varphi ,W}\leq C_{\varphi ,W}\left\Vert
\left( \sum\limits_{l=0}^{\infty }\left\vert \triangledown _{l}\right\vert
^{2}\right) ^{1/2}\right\Vert _{\varphi ,W}
\end{equation*}%
where $f\left( \cdot \right) \sim \sum\limits_{j=0}^{\infty }A_{j}\left(
\cdot ,f\right) $ and%
\begin{equation}
\triangledown _{l}:=\triangledown _{l}\left( \cdot ,f\right)
:=\sum\limits_{\nu =2^{l-1}}^{2^{l}-1}A_{\nu }\left( \cdot ,f\right) \text{%
,\quad }A_{1/2}\left( \cdot ,f\right) :=0.  \label{kla}
\end{equation}%
Refined inverse inequality of Marchaud type is also obtained (see \textsc{%
Corollary} \ref{corlll}): Suppose that $\varphi \in Y\left[ p,q\right] $ for
some $1<p,q<\infty $ and $\varphi \in \Delta _{2}$, $\varphi ^{\theta }$ is
quasiconvex for some $\theta \in \left( 0,1\right) $. If $\beta :=\max
\left\{ 2,q\right\} $ and $k\in \mathbb{R}^{+}$, $W\in A_{p\left( \varphi
\right) }$, $f\in L_{W}^{\varphi },$ then, there exists a\emph{\ }positive
constant $C_{l,k,W,\varphi }$ depending only on $l$, $k$, $\varphi $ and $W$
such that

\begin{equation*}
t^{2k}\left\{ \int\limits_{t}^{1}\left[ \frac{\Omega _{l}\left( f,u\right)
_{\varphi ,W}}{u^{2k}}\right] ^{\beta }\frac{du}{u}\right\} ^{1/\beta }\leq
C_{l,k,W,\varphi }\Omega _{k}\left( f,t\right) _{\varphi ,W}
\end{equation*}%
holds for $k,l\in \mathbb{R}^{+}$, $k<l$ and $0<t\leq 1/2$.

Jackson's second type direct theorem is a generalization of the inequality (%
\ref{dirr}) given in \textsc{Corollary} \ref{secotype}: Under the conditions
(\ref{kosul}), $\alpha \in \mathbb{R}^{+}$ and $f^{\left( \alpha \right)
}\in L_{W}^{\varphi }$, there exists a constant $C_{k,W,\varphi }\in \mathbb{%
R}^{+}$\ depending only on $k$, $\varphi $ and $W$\ such that%
\begin{equation*}
n^{\alpha }E_{n}\left( f\right) _{\varphi ,W}\leq C_{k,W,\varphi }\Omega
_{k}\left( f^{\left( \alpha \right) },\frac{1}{n}\right) _{\varphi ,W}
\end{equation*}%
holds for every $n\in \mathbb{N}$.

Another important property of the modulus of smoothness $\Omega _{k}\left(
f,\cdot \right) _{\varphi ,W}$ is that equivalence with Realization
functional%
\begin{equation*}
R_{l}\left( f,1/n,L_{W}^{\varphi }\right) :=\left\Vert f-t_{n}^{\ast
}(f)\right\Vert _{\varphi ,W}+n^{-l}\left\Vert \left( t_{n}^{\ast
}(f)\right) ^{\left( l\right) }\right\Vert _{\varphi ,W},
\end{equation*}%
where $n\in \mathbb{N}$, $l\in \mathbb{R}^{+}$, $t_{n}^{\ast }(f)\in 
\mathcal{T}_{n}$ be such that $\left\Vert f-t_{n}^{\ast }(f)\right\Vert
_{\varphi ,W}\leq CE_{n}\left( f\right) _{\varphi ,W}$. This equivalence is
given in (\ref{real}): Under the conditions (\ref{kosul}) we have%
\begin{equation*}
\Omega _{k}\left( f,1/n\right) _{\varphi ,W}\approx R_{2k}\left(
f,1/n,L_{W}^{\varphi }\right) \text{,\quad }n\in \mathbb{N}\text{,}
\end{equation*}%
where the equivalence constants are depend only on $k,\varphi $ and $W.$

The last equivalence with Realization functional yields the equivalence of
the fractional order weighted modulus of smoothness $\Omega _{k}\left(
f,\cdot \right) _{\varphi ,W}$ with the classical weighted Peetre's $K$%
-functional%
\begin{equation}
K_{l}\left( f,t,L_{W}^{\varphi }\right) :=\underset{g\in W_{\varphi ,W}^{l}}{%
\inf }\left\{ \left\Vert f-g\right\Vert _{\varphi ,W}+t^{l}\left\Vert
g^{\left( l\right) }\right\Vert _{\varphi ,W}\right\} \text{,\quad }t>0\text{%
,}  \label{kfunct}
\end{equation}%
where $W_{\varphi ,W}^{l}$ is the corresponding Sobolev space.

We obtained in (\ref{Kf}) that if the conditions (\ref{kosul}) are hold, then%
\begin{equation*}
\Omega _{k}\left( f,\delta \right) _{\varphi ,W}\approx K_{2k}\left(
f,\delta ,L_{W}^{\varphi }\right) \text{,\quad }\delta >0
\end{equation*}%
with the equivalence constants are depend only on $k,\varphi $ and $W.$ Here
and in what follows, $A\lesssim B$ means that there exists a constant $C$,
independent of essential parameters, such that the inequality $A\leq CB$ is
hold. If $A\lesssim B$ and $B\lesssim A$ simultaneously, we will write $%
A\approx B$. Throughout this work by $C_{a,b,c,...}$ we denote positive
constants, depending only on the parameters $a,b,c,...$, which can be
different values at different places.

In \S 2 basic properties of weighs and weighted Orlicz spaces. \S 3 includes
informations about three types transference results. In \S 4 we formulate
basic properties of the modulus of smoothness and a Marcinkiewicz type
interpolation theorem. \S 5 contains Jackson type direct theorem for
functions having Weyl's fractional derivative of degree $\alpha \geq 0$. In 
\S 6 we give some refinements of the Jackson inequality. Also we prove an
extrapolation theorem, Marcinkiewicz type multiplier theorem and
Littlewood-Paley type theorem. Realization theorem and an equivalence of the
weighted Peetre's $K$-functional with the fractional order weighted modulus
of smoothness are proved in \S 7. The last section \S 8 contains proofs.

\section{Orlicz spaces with Muckenhoupt weights}

A function $\omega :\mathsf{T}\mathbb{\rightarrow }\left[ 0,\infty \right] $
will be called weight if $\omega $ is measurable and positive a.e. on $%
\mathsf{T}.$ A 2$\pi $-periodic weight $\omega $ belongs to the Muckenhoupt
class $A_{p}$, $1\leq p<\infty $, if%
\begin{equation}
\left( meas(J)\right) ^{-1}\omega \left( J\right) \lesssim Cessinf%
_{x\in J}\omega \left( x\right) \text{, a.e. on }\mathsf{T}\text{%
,\quad }\left( p=1\right) \text{,}  \label{a1123}
\end{equation}%
\begin{equation}
C:=\sup\limits_{J}\frac{\omega \left( J\right) }{meas(J)}\left( \frac{\left[
\omega \left( J\right) \right] ^{\frac{1}{1-p}}}{meas(J)}\right)
^{p-1}<\infty \text{,\quad }\left( 1<p<\infty \right)   \label{ap}
\end{equation}%
with some finite constant independent of $J$, where $J$ is any subinterval
of $\mathsf{T}$ and $\omega (A):=\int_{A}\omega (t)dt$ for $A\subset \mathsf{%
T}$. The least constant $C$ in (\ref{a1123}) and (\ref{ap}) will be denoted
by $\left[ \omega \right] _{p}$ for $1\leq p<\infty .$

Let $\varphi $ be a quasiconvex function. We denote by $L_{W}^{\varphi }$
the class of Lebesgue measurable functions $f:\mathsf{T}\rightarrow \mathbb{R%
}$ satisfying the condition%
\begin{equation*}
\int_{\mathsf{T}}\varphi \left( C\left\vert f\left( x\right) \right\vert
\right) W\left( x\right) dx<\infty
\end{equation*}%
for some $C\in \mathbb{R}^{+}$. The collection of functions in $%
L_{W}^{\varphi }$ becomes a normed space with the Orlicz\textit{\ }norm%
\begin{equation*}
\left\Vert f\right\Vert _{\varphi ,W}:=\sup \left\{ \int\limits_{\mathsf{T}%
}\left\vert f\left( x\right) g\left( x\right) \right\vert W\left( x\right)
dx:\int\limits_{\mathsf{T}}\tilde{\varphi}\left( \left\vert g\right\vert
\right) W\left( x\right) dx\leq 1\right\}
\end{equation*}%
where%
\begin{equation}
\tilde{\varphi}\left( y\right) :=\sup_{x\geq 0}\left\{ xy-\varphi \left(
x\right) \right\} ,\quad y\geq 0,  \label{comp}
\end{equation}%
is the \textit{complementary function} of $\varphi $. The Banach space $%
L_{W}^{\varphi }$ is called \textit{Orlicz space}. For a weight $W$, taking $%
\varphi _{p}:=\varphi \left( x,p\right) :=x^{p}$, $1\leq p<\infty $, we
obtain the classical weighted Lebesgue space $L_{W}^{p}:=L_{W}^{\varphi
_{p}} $. If $\varphi $ is quasiconvex and $\tilde{\varphi}$ is its
complementary function, then \textit{Young's inequality} holds:%
\begin{equation}
xy\leq \varphi \left( x\right) +\tilde{\varphi}\left( y\right) \text{,\quad }%
x,y\geq 0\text{.}  \label{Young}
\end{equation}%
We define the \textit{Kolmogorov-Minkowski }functional as%
\begin{equation*}
\left\Vert f\right\Vert _{\left( \varphi \right) ,\omega }:=\inf \left\{
\tau >0:\int\limits_{\mathsf{T}}\varphi \left( \frac{\left\vert f\left(
x\right) \right\vert }{\tau }\right) W\left( x\right) dx\leq 1\right\} \text{%
.}
\end{equation*}%
There exist (\cite[p. 23]{GGKK}) constants $c$ and $C>0$ such that%
\begin{equation}
c\left\Vert f\right\Vert _{\left( \varphi \right) ,W}\leq \left\Vert
f\right\Vert _{\varphi ,W}\leq C\left\Vert f\right\Vert _{\left( \varphi
\right) ,W}\text{.}  \label{fac}
\end{equation}

A Young function $\varphi $ is said to be satisfy (global) $\Delta _{2}$ 
\textit{condition} if there is a constant $C\in \mathbb{R}^{+}$ such that $%
\varphi \left( 2x\right) \leq C\varphi \left( x\right) $ for all $x\in 
\mathbb{R}$. We will denote by $Q$ the class of functions $g$ satisfying $%
\Delta _{2}$ condition such that $g^{\theta }$ is quasiconvex for some $%
\theta \in \left( 0,1\right) $. For a quasiconvex function $\varphi $ we
give the definition of (\cite{gk94}) the indice $p\left( \varphi \right) $
of $\varphi $ as in (\ref{indic}). We set, as usual,%
\begin{equation*}
\left( p\left( \varphi \right) \right) ^{\prime }:=\frac{p\left( \varphi
\right) }{p\left( \varphi \right) -1}\text{ for }p\left( \varphi \right) >1%
\text{ and }W_{\ast }:=W^{1-\left( p\left( \varphi \right) \right) ^{\prime
}}\text{.}
\end{equation*}%
It is known that, if $W\in A_{p\left( \varphi \right) }$, then, $W_{\ast
}\in A_{\left( p\left( \varphi \right) \right) ^{\prime }}$.

We have the following variant of H\"{o}lder's inequality:

\begin{theorem}
\label{lbir}If $\varphi \in Q$, $W\in A_{p\left( \varphi \right) }$, $f\in
L_{W}^{\varphi }$, and $g\in L_{W_{\ast }}^{\tilde{\varphi}}$, then, there
exist an $v>1$ and a constant $c>1$, depending only on $\varphi $,$W$, such
that%
\begin{equation}
L^{\infty },C\left( \mathsf{T}\right) \hookrightarrow L_{W}^{\varphi
}\hookrightarrow L^{v}\hookrightarrow L^{1}\text{\quad and}  \label{cvf}
\end{equation}%
\begin{equation}
\int\limits_{\mathsf{T}}\left\vert f\left( x\right) g\left( x\right)
\right\vert dx\leq c\left\Vert f\right\Vert _{\varphi ,W}\left\Vert
g\right\Vert _{\tilde{\varphi},W_{\ast }}  \label{Hld}
\end{equation}%
hold.
\end{theorem}

If $\varphi \in Q$, $W\in A_{p\left( \varphi \right) }$, then the
Hardy-Littlewood Maximal function $M$ is a norm bounded (\cite[Theorem I]%
{gk94}) operator in $L_{W}^{\varphi }$. Using this it is obtained the
following lemma.

\begin{lemma}
\label{l2}(\cite{ai11})If $\varphi \in Q$, $W\in A_{p\left( \varphi \right)
} $, $f\in L_{W}^{\varphi }$, then there is a positive constant $c_{\varphi
,W} $, depending only on $\varphi $ and $W$, such that%
\begin{equation}
\left\Vert \mathcal{A}_{h}f\right\Vert _{\varphi ,W}\leq c_{\varphi
,W}\left\Vert f\right\Vert _{\varphi ,W}  \label{bou}
\end{equation}%
holds for any $h>0$.
\end{lemma}

\section{Some Transference Results}

We will mention here three transference results. We start with a
transference result obtained in \cite{akgArx} for weighted Lebesgue space $%
L_{W}^{p}$ with $1\leq p<\infty $, $W\in A_{p}$. We denote by $C^{\infty }$,
collection of continuous functions $g$ on $\mathsf{T}$ such that $g$ is
differentiable with arbitrary order on $\mathsf{T}$. Also let $C\left( 
\mathsf{T}\right) $ be class of continuous functions on $\mathsf{T}$.

\begin{definition}
\label{dd1}(\cite{akgArx})Suppose $1\leq p<\infty $, $W\in A_{p}$, $f\in
L_{W}^{p}$,%
\begin{equation*}
p^{\prime }:=\left\{ 
\begin{tabular}{cc}
$\frac{p}{p-1}$ & for $p>1,$ \\ 
$\infty $ & for $p=1,$%
\end{tabular}%
\right. \quad W^{\prime }:=\left\{ 
\begin{tabular}{cc}
$W^{1-p^{\prime }}$ & for $p>1,$ \\ 
$1$ & for $p=1.$%
\end{tabular}%
\right.
\end{equation*}%
For an $G\in L_{p^{\prime },W^{\prime }}\cap C^{\infty }$,\quad $\left\Vert
G\right\Vert _{p^{\prime },W^{\prime }}\leq 1$ we define%
\begin{equation}
F_{f,G}\left( u\right) :=\int\nolimits_{\mathsf{T}}f\left( x+u\right)
\left\vert G\left( x\right) \right\vert dx,\quad u\in \mathsf{T}\text{.}
\label{fbb}
\end{equation}
\end{definition}

\begin{theorem}
\label{Fu}(\cite{akgArx})If $1\leq p<\infty $, $W\in A_{p}$ and $f\in
L_{W}^{p}$, then the function $F_{f,G}\left( u\right) $, defined in (\ref%
{fbb}), is uniformly continuous on $\mathsf{T}.$
\end{theorem}

\begin{theorem}
\label{tra}(\cite{akgArx})Let $1\leq p<\infty $, $W\in A_{p}$, $f$, $g\in
L_{W}^{p}$. In this case, if%
\begin{equation*}
\left\Vert F_{g,G}\right\Vert _{C\left( \mathsf{T}\right) }\leq \mathbf{C}%
\left\Vert F_{f,G}\right\Vert _{C\left( \mathsf{T}\right) }
\end{equation*}%
holds for some absolute constant $\mathbf{C}$, then%
\begin{equation}
\left\Vert g\right\Vert _{p,W}\leq 24\mathbf{c}_{1}\mathbf{C}\left\Vert
f\right\Vert _{p,W}  \label{rrr}
\end{equation}%
with some constant $\mathbf{c}_{1}>0$ depending only on $p$,$W$.
\end{theorem}

Using the same steps with the constructions in Defition \ref{dd1}, Theorems %
\ref{Fu}-\ref{tra} of \cite{akgArx} we can obtain the following similar
results for functions in $L_{W}^{\varphi }$ where $\varphi \in Q$, $W\in
A_{p\left( \varphi \right) }$. We denote by $S\left( \mathsf{T}\right) $ the
collection of simple functions on $\mathsf{T}$.

\begin{definition}
\label{dd1x}Suppose that $\varphi \in Q$, $W\in A_{p\left( \varphi \right) }$%
, and $f\in L_{W}^{\varphi }$. For an $G\in L_{\tilde{\varphi},W_{\ast
}}\cap S\left( \mathsf{T}\right) $,\quad $\left\Vert G\right\Vert _{\tilde{%
\varphi},W_{\ast }}\leq 1$ we define%
\begin{equation}
\Xi _{f,G}\left( u\right) :=\int\nolimits_{\mathsf{T}}f\left( x+u\right)
\left\vert G\left( x\right) \right\vert dx,\quad u\in \mathsf{T}\text{.}
\label{dd1xy}
\end{equation}
\end{definition}

\begin{theorem}
\label{FuX}If $\varphi \in Q$, $W\in A_{p\left( \varphi \right) }$, and $%
f\in L_{W}^{\varphi }$, then the function $\Xi _{f,G}\left( u\right) $,
defined in (\ref{dd1xy}), is uniformly continuous on $\mathsf{T}.$
\end{theorem}

\begin{theorem}
\label{traX}Let $\varphi \in Q$, $W\in A_{p\left( \varphi \right) }$, and $%
f,g\in L_{W}^{\varphi }$. In this case, if%
\begin{equation*}
\left\Vert \Xi _{g,G}\right\Vert _{C\left( \mathsf{T}\right) }\leq
C\left\Vert \Xi _{f,G}\right\Vert _{C\left( \mathsf{T}\right) }
\end{equation*}%
holds for some absolute constant $C$, then%
\begin{equation}
\left\Vert g\right\Vert _{\varphi ,W}\leq \mathbf{c}_{2}C\left\Vert
f\right\Vert _{\varphi ,W}  \label{vbn}
\end{equation}%
with some constant $\mathbf{c}_{2}>0$ depending only on $\varphi $,$W$.
\end{theorem}

In some cases we will use another well known transference result: the
Marcinkiewicz Interpolation Theorem. The following Marcinkiewicz type
interpolation theorem that was proved in more general form in \cite[Theorem
3.6]{gk95}, for an subadditive operator $T$, mapping the measure space $%
\left( Y_{0},S_{0},\upsilon _{0}\right) $ into measure space $\left(
Y_{1},S_{1},\upsilon _{1}\right) $.

\begin{theorem}
\label{l1}(\cite[Theorem 3.6]{gk95})Let $\varphi $ be a quasiconvex Young
function and $\tilde{\varphi}$ be its complementary (\ref{comp}). Suppose
that $1\leq r<p\left( \varphi \right) \leq \left( p\left( \tilde{\varphi}%
\right) \right) ^{\prime }<s<\infty $. If there exist constants $c$ and $C>0$
such that for all $f\in L^{r}+L^{s}$%
\begin{equation*}
\int\limits_{\left\{ y\in Y_{1}:\left\vert Tf\left( x\right) \right\vert
>\lambda \right\} }d\nu _{1}\leq c\lambda ^{-r}\int\limits_{Y_{0}}\left\vert
f\left( x\right) \right\vert ^{r}d\nu _{0}
\end{equation*}%
and%
\begin{equation*}
\int\limits_{\left\{ y\in Y_{1}:\left\vert Tf\left( x\right) \right\vert
>\lambda \right\} }d\nu _{1}\leq C\lambda ^{-s}\int\limits_{Y_{0}}\left\vert
f\left( x\right) \right\vert ^{s}d\nu _{0}\text{,}
\end{equation*}%
then%
\begin{equation*}
\int\limits_{Y_{1}}\varphi \left( Tf\left( x\right) \right) d\nu
_{1}\lesssim \int\limits_{Y_{0}}\varphi \left( f\left( x\right) \right) d\nu
_{0}.
\end{equation*}
\end{theorem}

Third transference result is celebrated Extrapolation Theorem. In this work
we proved an Extrapolation Theorem \ref{ET} designed for Orlicz spaces
generated by quasiconvex Young functions.

\section{Modulus of smoothness}

First we start with an extention of Lemma \ref{l2}. To do this we will
employ transference Theorem \ref{traX}.

\begin{lemma}
\label{l2x}Let $i\in \mathbb{N}$, $\varphi \in Q$, $W\in A_{p\left( \varphi
\right) }$, and $f\in L_{W}^{\varphi }$. Then, there is a positive constant $%
c_{\varphi ,W}$, depending only on $\varphi $,$W$ \textbf{but independent of 
}$i$, such that%
\begin{equation}
\left\Vert \left( \mathcal{A}_{h}\right) ^{j}f\right\Vert _{\varphi ,W}\leq
c_{\varphi ,W}\left\Vert f\right\Vert _{\varphi ,W}  \label{ahashI}
\end{equation}%
holds for any $h>0$.
\end{lemma}

Note that in Lemma \ref{l2x} constant in (\ref{ahashI}) does not depend on
parameter $j$. This fact is important to construct binomial expansion of $I-%
\mathcal{A}_{h}$ with exponent $k\in \mathbb{R}^{+}$.

Now we can define on $\mathsf{T}$, as usual

\begin{equation}
\left( I-\mathcal{A}_{h}\right) ^{k}f\left( \cdot \right)
=\sum\limits_{j=0}^{\infty }\frac{\left( -1\right) ^{j}\Gamma \left(
k+1\right) }{\Gamma \left( j+1\right) \Gamma (k-j+1)}\left( \mathcal{A}%
_{h}\right) ^{j}f\left( \cdot \right)  \label{bina}
\end{equation}%
where\ $k\in \mathbb{R}^{+}$, $\varphi \in Q$, $W\in A_{p\left( \varphi
\right) }$, $f\in L_{W}^{\varphi }$, $\Gamma $ is Gamma function, $I$ is the
identity operator.

\begin{lemma}
\label{l2xx}Suppose that $\varphi \in Q$, $W\in A_{p\left( \varphi \right) }$%
, $f\in L_{W}^{\varphi }$ and $k\in \mathbb{R}^{+}$. Then, the series (\ref%
{bina}) is convergent in $\left\Vert \cdot \right\Vert _{\varphi ,W}$ norm
and%
\begin{equation}
\left\Vert \left( I-\mathcal{A}_{h}\right) ^{k}f\right\Vert _{\varphi
,W}\lesssim \left\Vert f\right\Vert _{\varphi ,W}.  \label{AX}
\end{equation}
\end{lemma}

\begin{definition}
Based on the last inequality we can define weighted modulus of smoothness of
fractional order $k\in \mathbb{R}^{+}\cup \left\{ 0\right\} $ as%
\begin{equation}
\Omega _{k}\left( f,\delta \right) _{\varphi ,W}\text{:=}\underset{0<h\leq
\delta }{\sup }\left\Vert \left( I-\mathcal{A}_{h}\right) ^{k}f\right\Vert
_{\varphi ,W}\text{,\quad }k\in \mathbb{R}^{+}\text{ ; }\Omega _{0}\left(
f,\delta \right) _{\varphi ,W}\text{:=}\left\Vert f\right\Vert _{\varphi ,W}
\label{modul}
\end{equation}%
where $\varphi \in Q$, $W\in A_{p\left( \varphi \right) }$ and $f\in
L_{W}^{\varphi }$.
\end{definition}

For a given $\varphi \in Q$, $W\in A_{p\left( \varphi \right) }$ we can
define the corresponding real trigonometric Fourier series%
\begin{equation}
f(x)\sim \sum\limits_{j=0}^{\infty }\left( a_{j}\left( f\right) \cos
jx+b_{j}\left( f\right) \sin jx\right) :=\sum\limits_{j=0}^{\infty
}A_{j}\left( x,f\right)  \label{fo}
\end{equation}%
of $f\in L_{W}^{\varphi }$ where%
\begin{equation*}
a_{j}\left( f\right) :=\frac{1}{\pi }\int\nolimits_{\mathsf{T}}f\left(
x\right) \cos jxdx\text{ }\left( j\in \left\{ 0\right\} \cup \mathbb{N}%
\right) \text{, }b_{j}\left( f\right) =\frac{1}{\pi }\int\nolimits_{\mathsf{T%
}}f\left( x\right) \sin jxdx\text{ }\left( j\in \mathbb{N}\right) .
\end{equation*}%
Let $S_{n}\left( f\right) :=S_{n}\left( x,f\right)
:=\sum\limits_{j=0}^{n}A_{j}\left( x,f\right) $, $(n\in \left\{ 0\right\}
\cup \mathbb{N})$ be the $n$th partial sum of the Fourier series (\ref{fo})
of $f$.

Using Interpolation Theorem \ref{l1} (with $d\nu _{1}=d\nu _{0}=W\left(
x\right) dx$); Theorems 1 and 8 of \cite{hmw72} and \cite[(4.3), (4.4) of p.
111]{z2}, if $\varphi \in Q$, $W\in A_{p\left( \varphi \right) }$, $f\in
L_{W}^{\varphi }$, then the partial sum operator $S_{n}:L_{W}^{\varphi
}\rightarrow L_{W}^{\varphi }$ ($f\longmapsto S_{n}f$ ) and conjugate
operator $B:L_{W}^{\varphi }\rightarrow L_{W}^{\varphi }$ ($f\longmapsto 
\tilde{f}$ ) are norm bounded in $L_{W}^{\varphi }$, namely,%
\begin{equation}
\left\Vert S_{n}\left( f\right) \right\Vert _{\varphi ,W}\lesssim \left\Vert
f\right\Vert _{\varphi ,W}\text{,\quad }\left\Vert Bf\right\Vert _{\varphi
,W}\lesssim \left\Vert f\right\Vert _{\varphi ,W}.  \label{esen}
\end{equation}%
Therefore we get%
\begin{equation*}
\left\Vert f-S_{n}\left( f\right) \right\Vert _{\varphi ,W}\lesssim
E_{n}\left( f\right) _{\varphi ,W}
\end{equation*}%
for $n\in \left\{ 0\right\} \cup \mathbb{N}$ and $\varphi \in Q$, $W\in
A_{p\left( \varphi \right) }$, $f\in L_{W}^{\varphi }.$

Taking $\varphi \in Q$, $W\in A_{p\left( \varphi \right) }$, then by the
Lemma 3 of \cite{K02}, the set of trigonometric polynomials is a dense
subset of $L_{W}^{\varphi }$. Then the approximation problems make sense in $%
L_{W}^{\varphi }$ and $E_{n}\left( f\right) _{\varphi ,W}\searrow 0$ as $%
n\rightarrow \infty $. The corresponding Fourier series (\ref{fo}) of $f\in
L_{W}^{\varphi }$ converges in $L_{W}^{\varphi }$ norm for $\varphi \in Q$
and $W\in A_{p\left( \varphi \right) }.$

\section{jackson type inequalities of approximation}

The following Lemma has a key role for obtaining the direct Theorem \ref%
{teo2}.

\begin{lemma}
\label{y}Suppose that $\varphi \in Q$, $W\in A_{p\left( \varphi \right) }$, $%
f\in L_{W}^{\varphi }$ and $n,m,k\in \mathbb{N}$. Then, for any $h>0$, there
holds%
\begin{equation*}
\left\Vert (I-\mathcal{A}_{h})^{k}f\right\Vert _{\varphi ,W}\lesssim \frac{C%
}{2^{mk}}\Vert f\Vert _{\varphi ,W}+C^{\prime }\left\Vert (I-\mathcal{A}%
_{h})^{k+1}f\right\Vert _{\varphi ,W}
\end{equation*}%
where constant $C\in \mathbb{R}^{+}$ depends only on $k$,$\varphi $,$W$ and
constant $C^{\prime }\in \mathbb{R}^{+}$ depends only on $m$,$\varphi $,$W$.
\end{lemma}

The following Jackson type theorem relates the best approximation error $%
E_{n}\left( f\right) _{\varphi ,W}$ with the modulus of smoothness $\Omega
_{k}\left( f,n^{-1}\right) _{\varphi ,W}$.

\begin{theorem}
\label{teo2}If $\varphi \in Q$, $W\in A_{p\left( \varphi \right) }$, $f\in
L_{W}^{\varphi }$ and $k\in \mathbb{R}^{+}$, then%
\begin{equation}
E_{n}\left( f\right) _{\varphi ,W}\lesssim \Omega _{k}\left( f,\frac{1}{n}%
\right) _{\varphi ,W}  \label{JT}
\end{equation}%
holds for $n\in \mathbb{N}$ with some constant depending only on $k,W$ and $%
\varphi .$
\end{theorem}

If $W\equiv 1$, $\varphi \left( x\right) =x^{p}$, $1\leq p<\infty $, on more
general Homogenous Banach spaces this theorem was obtained in \cite{dh}.
When $1<p<\infty $, $W\in A_{p}$, $\varphi \left( x\right) =x^{p}$, $k=1$, (%
\ref{JT}) was proved in \cite{Ky1,Ky2}. For $k=1$ see the papers \cite%
{ai11,akg12}. Under the conditions $k=1$ and convex $\varphi $ some variant
of (\ref{JT}) was obtained in \cite{ig06}.

We say that a function $f\in L_{W}^{\varphi },$ ($\varphi \in Q$, $W\in
A_{p\left( \varphi \right) }$), has a Weyl's derivative $f^{\left( \alpha
\right) }$ of degree $\alpha \in \mathbb{R}^{+}$ if the series%
\begin{equation*}
\sum\limits_{k=1}^{\infty }k^{\alpha }\left( a_{k}\left( f\right) \cos
k\left( x+\frac{\alpha \pi }{2k}\right) +b_{k}\left( f\right) \sin k\left( x+%
\frac{\alpha \pi }{2k}\right) \right)
\end{equation*}%
is the Fourier series of function $f^{\left( \alpha \right) }.$

Let $W_{\varphi ,W}^{\alpha }$, $\alpha \in \mathbb{R}^{+}$ be the class of
functions $f\in L_{W}^{\varphi }$ such that $f^{\left( \alpha \right) }\in
L_{W}^{\varphi }$. Let $\alpha \in \mathbb{R}^{+}$ and $f\in W_{\varphi
,W}^{\alpha }$, then for every $n\in \mathbb{N}$\ there exist (\cite[Theorem
1]{ai11}) positive constants, independent of $n,$\ such that%
\begin{equation}
E_{n}\left( f\right) _{\varphi ,W}\lesssim \frac{1}{n^{\alpha }}E_{n}\left(
f^{\left( \alpha \right) }\right) _{\varphi ,W}.  \label{so}
\end{equation}%
Using the inequality (\ref{so}) and the inequality (\ref{JT}) we have the
following second type Jackson theorem.

\begin{theorem}
\label{secotype}Let $\alpha ,k\in \mathbb{R}^{+}\cup \left\{ 0\right\} $, $%
\varphi \in Q$, $W\in A_{p\left( \varphi \right) }$ and $f\in W_{\varphi
,W}^{\alpha }$, then for every $n\in \mathbb{N}$\ there exists a constant,
independent of $n$,\ such that%
\begin{equation*}
E_{n}\left( f\right) _{\varphi ,W}\lesssim \frac{1}{n^{\alpha }}\Omega
_{k}\left( f^{\left( \alpha \right) },\frac{1}{n}\right) _{\varphi ,W}
\end{equation*}%
holds.
\end{theorem}

When $k=1$ the last inequality was proved in \cite{ai11,akg12}. For $k=1$, $%
\alpha \in \mathbb{N}$ and convex $\varphi $ some variant of the last
inequality was obtained in \cite{ig06}.

\section{Refinements of Jackson type inequality and extrapolation}

After the inequality (\ref{JT}) we obtained a refinement of (\ref{JT}).

\begin{theorem}
\label{teo4} Under the conditions of Theorem \ref{teo2} we have the
inequality%
\begin{equation}
\left( \prod\limits_{j=1}^{n}E_{j}\left( f\right) _{\varphi ,W}\right)
^{1/n}\lesssim \Omega _{k}\left( f,\frac{1}{n}\right) _{\varphi ,W}
\label{IJ}
\end{equation}%
with some constant depending only on $\varphi ,k$ and $W$.
\end{theorem}

\begin{remark}
\label{rem2} The inequality (\ref{IJ}) is never worse than the inequality (%
\ref{JT}), because using $E_{n}\left( f\right) _{\varphi ,W}\searrow 0$ as $%
n\rightarrow \infty $, we get%
\begin{equation*}
E_{n}\left( f\right) _{\varphi ,W}\leq \left(
\prod\limits_{j=1}^{n}E_{j}\left( f\right) _{\varphi ,W}\right) ^{1/n}\leq
C_{\varphi ,W,k}\Omega _{k}\left( f,\frac{1}{n}\right) _{\varphi ,W}
\end{equation*}%
where the constant $C_{\varphi ,W,k}$ depend only on $\varphi ,W$ and $k$.
On the other hand, in some cases, the inequality (\ref{IJ}) gives better
results than that of inequality (\ref{JT}). For example, if $E_{n}\left(
f\right) _{\varphi ,W}=2^{-n}$, then inequality (\ref{JT}) gives $\Omega
_{k}\left( f,\frac{1}{n}\right) _{\varphi ,W}\geq C_{\varphi ,k,W}2^{-n}$
but inequality (\ref{IJ}) yields $\Omega _{k}\left( f,\frac{1}{n}\right)
_{\varphi ,W}\geq C_{\varphi ,k,W}2^{-n/2}.$
\end{remark}

In $L^{\infty }$ Theorem \ref{teo4} was obtained (\cite{TiNa}) in 1979 by
Natanson and M. Timan for $k\in \mathbb{N}$ and classical moduli of
continuity.

\begin{theorem}
\label{proc}Let $\varphi \in Y\left[ p,q\right] $ for some $1<p,q<\infty $, $%
\varphi \in Q$, $W\in A_{p\left( \varphi \right) }$ and $f\in L_{W}^{\varphi
}$. If $n\in \mathbb{N}$, $k\in \mathbb{R}^{+}$ and $\beta :=\max \left\{
2,q\right\} $, then there exists a positive constant depending only on $k$\
and $\varphi ,W$ such that%
\begin{equation}
\frac{1}{n^{2k}}\left\{ \sum\limits_{\nu =1}^{n}\nu ^{2\beta k-1}E_{\nu
-1}^{\beta }\left( f\right) _{\varphi ,W}\right\} ^{1/\beta }\lesssim \Omega
_{k}\left( f,\frac{1}{n}\right) _{\varphi ,W}  \label{impj}
\end{equation}%
holds.
\end{theorem}

Under the conditions $W\equiv 1$, $\varphi \left( x\right) =x^{p}$, $%
1<p<\infty $, a (\ref{impj})-type inequality was proved first in \cite{TiM1}
for the classical modulus of smoothness. In case $k=1$ see also \cite{A,acc}.

Before the proof of Theorem \ref{proc} we need the following extrapolation
theorem.

\begin{theorem}
\label{ET} Let $\varphi \in \Delta _{2}$ be quasiconvex and $\varphi \left(
t^{1/p_{0}}\right) $ be a convex function for some $p_{0}\in \left( 1,\infty
\right) $. We suppose that $\mathcal{F}$ is a family of couples of
nonnegative functions such that%
\begin{equation*}
\int\limits_{\mathsf{T}}F\left( x\right) ^{p_{0}}W\left( x\right) dx\lesssim
\int\limits_{\mathsf{T}}g\left( x\right) ^{p_{0}}W\left( x\right) dx\text{%
,\quad }\left( f,g\right) \in \mathcal{F}
\end{equation*}%
for all $W\in A_{1}$ provided the left hand side is finite. Then, for all $%
W\in A_{p\left( \varphi \right) }$, the inequality%
\begin{equation*}
\int\limits_{\mathsf{T}}\varphi \left( F\left( x\right) \right) W\left(
x\right) dx\lesssim \int\limits_{\mathsf{T}}\varphi \left( g\left( x\right)
\right) W\left( x\right) dx
\end{equation*}%
holds for any $\left( F,g\right) \in \mathcal{F}$ when the left-hand side is
finite.
\end{theorem}

For $W\equiv 1$ this extrapolation theorem was obtained in \cite{k09}.

To prove Theorem \ref{proc} we will require also the Marcinkiewicz
multiplier theorem and the Littlewood-Paley type theorems: Using
interpolation Theorem \ref{l1} and Marcinkiewicz multiplier theorem (\cite[%
Theorems 1 and 2]{k80}) for Lebesgue spaces with Muckenhoupt weights (one
dimensional 2$\pi $-periodic versions) (see also \cite[Theorem 4.4, p.956]%
{bg03}) we obtain the following Marcinkiewicz multiplier theorem

\begin{theorem}
\label{mmt}Let a sequence $\left\{ \lambda _{l}\right\} $\ of real numbers
be satisfy the properties%
\begin{equation}
\left\vert \lambda _{l}\right\vert \leq A\text{,\quad }\sum%
\limits_{l=2^{m-1}}^{2^{m}-1}\left\vert \lambda _{l}-\lambda
_{l+1}\right\vert \leq A  \label{MS}
\end{equation}%
for all $l,m\in \mathbb{N}$, where $A$\ doesn't depend on $l$\ and $m$. If $%
\varphi \in Q$, $W\in A_{p\left( \varphi \right) }$, $f\in L_{W}^{\varphi }$%
,\ then there is a function $G\in L_{W}^{\varphi }$\emph{\ }such that the
series $\sum\nolimits_{k=0}^{\infty }\lambda _{k}A_{k}\left( \cdot ,f\right) 
$\ is Fourier series of $G$\ and%
\begin{equation}
\left\Vert G\right\Vert _{\varphi ,W}\lesssim \left\Vert f\right\Vert
_{\varphi ,W}  \label{mult}
\end{equation}%
holds with a positive constant,\ does not depend on $f$.
\end{theorem}

In case $W\equiv 1$, $\varphi \left( x\right) =x^{p}$, $1<p<\infty $, (\ref%
{mult}) was proved by Marcinkiewicz in Theorem 4.14 of \cite{z2}. For $W\in
A_{p}$, $\varphi \left( x\right) =x^{p}$, $1<p<\infty $, (\ref{mult}) was
obtained by Kurtz in \cite{k80} and Berkson-Gillespie \cite[Theorem 4.4,
p.956]{bg03}.

Littlewood-Paley type theorem is

\begin{theorem}
\label{lpt}If $\varphi \in \Delta _{2}$ is quasiconvex and $\varphi \left(
t^{1/p_{0}}\right) $ is a convex function for some $p_{0}\in \left( 1,\infty
\right) $, $W\in A_{p\left( \varphi \right) }$ and $f\in L_{W}^{\varphi }$,
then there are positive constants depending only on $\varphi ,W$ such that%
\begin{equation}
\left\Vert \left( \sum\limits_{\mu =0}^{\infty }\left\vert \triangledown
_{l}\right\vert ^{2}\right) ^{1/2}\right\Vert _{\varphi ,W}\approx
\left\Vert f\right\Vert _{\varphi ,W}  \label{L-P}
\end{equation}%
where $f\left( \cdot \right) \sim \sum\nolimits_{j=0}^{\infty }A_{j}\left(
\cdot ,f\right) $ and $\triangledown _{l}$ be as in (\ref{kla}).
\end{theorem}

In case $W\equiv 1$, $\varphi \left( x\right) =x^{p}$, $1<p<\infty $, (\ref%
{L-P}) was obtained by Littlewood and Paley in \cite{lp31}. For $W\in A_{p}$%
, $\varphi \left( x\right) =x^{p}$, $1<p<\infty $, (\ref{L-P}) was obtained
by Kurtz in \cite{k80} and Berkson-Gillespie \cite[Theorem 4.5, p.957]{bg03}.

\begin{corollary}
\label{corlll}Under the conditions of Theorems \ref{proc} and \ref{teo2} if $%
k,l\in \mathbb{R}^{+}$, $k<l$ and $0<t\leq 1/2$, then there exists\emph{\ }%
positive constant $C$ depending only on $l$, $k$ and $\varphi ,W$ such that%
\begin{equation*}
t^{2k}\left\{ \int\limits_{t}^{1}\left[ \frac{\Omega _{l}\left( f,u\right)
_{\varphi ,W}}{u^{2k}}\right] ^{\beta }\frac{du}{u}\right\} ^{1/\beta
}\lesssim \Omega _{k}\left( f,t\right) _{\varphi ,W}
\end{equation*}%
hold.
\end{corollary}

\section{Realization}

The following theorem includes the realization result and an equivalence of
weighted fractional order modulus of smoothness (\ref{modul}) and weighted
Peetre's $K$-functional (\ref{kfunct}).

\begin{theorem}
\label{rea}If $\varphi \in Q$, $W\in A_{p\left( \varphi \right) }$, $f\in
L_{W}^{\varphi }$ and $k\in \mathbb{R}^{+}$, then the equivalence%
\begin{equation}
\Omega _{k}\left( f,1/n\right) _{\varphi ,W}\approx R_{2k}\left(
f,1/n,L_{W}^{\varphi }\right)  \label{real}
\end{equation}%
holds for $n\in \mathbb{N}$, where the equivalence constants are depend only
on $k$ and $\varphi ,W$. Furthermore, we have 
\begin{equation}
\Omega _{k}\left( f,\delta \right) _{\varphi ,W}\approx K_{2k}\left(
f,\delta ,L_{W}^{\varphi }\right) \text{,\quad }\delta \geq 0,  \label{Kf}
\end{equation}%
where the equivalence constants are depend only on $k$ and $\varphi ,W$.
\end{theorem}

Equivalence (\ref{Kf}) was proved first in \cite{Gadj} for $k=1$, $%
1<p<\infty $, $W\in A_{p}$ and $\varphi \left( x\right) =x^{p}$. In the case 
$k=1$ (\ref{Kf}) was obtained in \cite{ai11,akg12}. Under the conditions $%
k=1 $ and convex $\varphi $, (\ref{JT}) was obtained in \cite{ig06}.

Equivalence (\ref{Kf}) implies the following result.

\begin{corollary}
\label{co}If $\varphi \in Q$, $W\in A_{p\left( \varphi \right) }$, $f\in
L_{W}^{\varphi }$, $k\in \mathbb{R}^{+}$, then 
\begin{equation*}
\Omega _{k}\left( f,\lambda \delta \right) _{\varphi ,W}\lesssim \left(
1+\lfloor \lambda \rfloor \right) ^{2k}\Omega _{k}\left( f,\delta \right)
_{\varphi ,W},\quad \delta ,\lambda \in \mathbb{R}^{+},
\end{equation*}%
and%
\begin{equation*}
\Omega _{k}\left( f,\delta \right) _{\varphi ,W}\delta ^{-2k}\lesssim \Omega
_{k}\left( f,\delta _{1}\right) _{\varphi ,W}\delta _{1}^{-2k},\quad
0<\delta _{1}\leq \delta ,
\end{equation*}%
where $\lfloor x\rfloor :=\sup \left\{ a\in \mathbb{N\cup }\left\{ 0\right\}
:a\leq x\right\} .$
\end{corollary}

We take a trigonometric polynomial $T\in \mathcal{T}_{n}$%
\begin{equation*}
T\left( x\right) =\sum\limits_{j=0}^{n}\left( a_{j}\cos jx+b_{j}\sin
jx\right) =\sum\limits_{j=0}^{n}A_{j}\left( x,T\right) ,\text{\quad }%
a_{j}\in \mathbb{R}\text{ }\left( j\in \left\{ 0\right\} \cup \mathbb{N}%
\right)
\end{equation*}%
and we define its conjugate $\widetilde{T}$ by%
\begin{equation*}
\widetilde{T}\left( x\right) =\sum\limits_{j=1}^{n}\left( a_{j}\sin
jx-b_{j}\cos jx\right) =:\sum\limits_{j=1}^{n}A_{j}\left( x,\widetilde{T}%
\right) .
\end{equation*}

\begin{lemma}
\label{lem2}Let $k\in \mathbb{R}^{+}$, $\varphi \in Q$, $W\in A_{p\left(
\varphi \right) }$ and $T_{n}\in \mathcal{T}_{n}$, $n\in \mathbb{N}$. Then%
\begin{equation}
\Omega _{k}\left( T_{n},\frac{1}{n}\right) _{\varphi ,W}\lesssim \frac{1}{%
n^{2k}}\left\Vert T_{n}^{(2k)}\right\Vert _{\varphi ,W}  \label{aha}
\end{equation}%
holds with some constant depending only on $k,\varphi $ and $W$.
\end{lemma}

(\ref{aha}) was proved in \cite{Gadj} for $k\in \mathbb{N}$, $1<p<\infty $, $%
W\in A_{p}$ and $\varphi \left( x\right) =x^{p}$. For $1<p<\infty $, $W\in
A_{p}$ and $\varphi \left( x\right) =x^{p}$ see \cite{Akgeja,akgStP}. In
case $1<p<\infty $, $W\in A_{p}$ and convex $\varphi \left( x\right) $ see
also \cite{ig06}.

\begin{lemma}
\label{lem3}Let $k\in \mathbb{R}^{+}$, $\varphi \in Q$, $W\in A_{p\left(
\varphi \right) }$ and $T_{n}\in \mathcal{T}_{n}$, $n\in \mathbb{N}$. Then%
\begin{equation*}
\frac{1}{n^{2k}}\left\Vert T_{n}^{(2k)}\right\Vert _{\varphi ,W}\lesssim
\Omega _{k}\left( T_{n},\frac{1}{n}\right) _{\varphi ,W}
\end{equation*}%
holds with some constant depending only on $k,\varphi $ and $W$.
\end{lemma}

\begin{lemma}
\label{kubu} Let $\varphi \in Q$, $W\in A_{p\left( \varphi \right) }$, $k\in 
\mathbb{R}^{+}$ and $f\in L_{\varphi ,W}^{2k}$. Then for any $0<t<\infty $,
the following inequality is hold%
\begin{equation*}
\Omega _{k}(f,t)_{\varphi ,W}\lesssim t^{2k}\Vert f^{\left( 2k\right) }\Vert
_{\varphi ,W}
\end{equation*}%
with some constant depending only on $k,\varphi $ and $W$.
\end{lemma}

\section{Proofs}

\begin{proof}[\textbf{Proof of Theorem \protect\ref{lbir}}]
\textbf{Proof of (\ref{cvf})}: Right embedding of (\ref{cvf}) is well-known.
Left embedding of (\ref{cvf}) is easy to obtain from Young's inequality (\ref%
{Young}) and (\ref{fac}).

We consider mid embedding in (\ref{cvf}). From $\varphi \in Q$, $W\in
A_{p\left( \varphi \right) }$, and $p\left( \varphi \right) >1$, there
exists a $v_{1}\in (1,p\left( \varphi \right) )$ such that $W\in A_{v_{1}}$.
Using definition of $p\left( \varphi \right) $ there exists a $v\in
(v_{1},p\left( \varphi \right) )$ so that $\varphi ^{1/v}$ is quasiconvex.
Hence (\cite[p.4, Lemma 1.1.1]{KK91})%
\begin{equation*}
\frac{\varphi ^{1/v}\left( t\right) }{t}
\end{equation*}%
is quasiincreasing, namely, the inequality%
\begin{equation*}
\frac{\varphi ^{1/v}\left( t\right) }{t}\leq \frac{C\varphi ^{1/v}\left(
Cs\right) }{s}
\end{equation*}%
holds for some $C>1$ where $0\leq t\leq s.$ For $\left\Vert f\right\Vert
_{\varphi ,W}\leq 1$ we get%
\begin{equation*}
\int\limits_{\left\{ \left\vert f\right\vert >1\right\} }\left\vert f\left(
x\right) \right\vert ^{v}W\left( x\right) dx<\frac{C^{2}}{\varphi \left(
1\right) }.
\end{equation*}%
Hence, a scaling argumet gives%
\begin{equation*}
\int\limits_{\mathsf{T}}\left\vert f\left( x\right) \right\vert ^{v}W\left(
x\right) dx\leq \left( \frac{C^{2}}{\varphi \left( 1\right) }+W\left( 
\mathsf{T}\right) \right) \left\Vert f\right\Vert _{\varphi ,W}^{v}
\end{equation*}%
and%
\begin{equation*}
\left\Vert f\right\Vert _{v,W}\leq \left( \frac{C^{2}}{\varphi \left(
1\right) }+W\left( \mathsf{T}\right) \right) ^{1/v}\left\Vert f\right\Vert
_{\varphi ,W}
\end{equation*}%
for $\varphi \in Q$, $W\in A_{p\left( \varphi \right) }$ and $f\in
L_{W}^{\varphi }$. Mid embedding of (\ref{cvf}) is proved.

\textbf{Proof of (\ref{Hld})}: (i) First we obtain for $f\in L_{W}^{\varphi
} $ that%
\begin{equation}
\int\limits_{\mathsf{T}}\left\vert f\left( x\right) \right\vert dx\leq
c\left\Vert f\right\Vert _{\varphi ,W}.  \label{h1}
\end{equation}

Assume $\left\Vert f\right\Vert _{\varphi ,W}\leq 1.$ Then,%
\begin{equation*}
\int\limits_{\mathsf{T}}\left\vert f\left( x\right) \right\vert dx\leq 2\pi
+\left( \int\limits_{\left\{ \left\vert f\right\vert >1\right\} }\left\vert
f\left( x\right) \right\vert ^{v}W\left( x\right) dx\right) ^{\frac{1}{v}%
}\left( \int\limits_{\left\{ \left\vert f\right\vert >1\right\} }W\left(
x\right) ^{-\frac{v^{\prime }}{v}}dx\right) ^{\frac{1}{v^{\prime }}}
\end{equation*}%
\begin{equation*}
\leq 2\pi +\left( \frac{C^{2}}{\varphi \left( 1\right) }\frac{\left( 2\pi
\right) ^{v}}{W\left( \mathsf{T}\right) }\left[ \omega \right] _{p}\right) ^{%
\frac{1}{v}}=:c_{0}.
\end{equation*}%
A scaling argument implies%
\begin{equation*}
\int\limits_{\mathsf{T}}\left\vert f\left( x\right) \right\vert dx\leq
c_{0}\left\Vert f\right\Vert _{\varphi ,W}.
\end{equation*}

(ii) We set $E_{0}$:=$\left\{ x\in \mathsf{T}\text{:}\left\vert f\right\vert
<1\text{ \ }\wedge \text{ \ }\left\vert g\right\vert <1\right\} $, $E_{1}$:=$%
\left\{ x\in \mathsf{T}\text{:}\left\vert f\right\vert <1\text{ \ }\wedge 
\text{ \ }\left\vert g\right\vert >1\right\} $,

$E_{2}$:=$\left\{ x\in \mathsf{T}\text{:}\left\vert f\right\vert >1\text{ \ }%
\wedge \text{ \ }\left\vert g\right\vert <1\right\} $, $E_{3}$:=$\left\{
x\in \mathsf{T}\text{:}\left\vert f\right\vert >1\text{ \ }\wedge \text{ \ }%
\left\vert g\right\vert >1\right\} $. Then%
\begin{equation*}
\int\limits_{\mathsf{T}}\left\vert f\left( x\right) g\left( x\right)
\right\vert dx=\sum_{i=0}^{3}\int\limits_{E_{i}}\left\vert f\left( x\right)
g\left( x\right) \right\vert dx.
\end{equation*}
(a) Clearly $\int\nolimits_{E_{0}}\left\vert f\left( x\right) g\left(
x\right) \right\vert dx\leq 2\pi .$ (b) Note that integral $%
\int\nolimits_{E_{2}}\left\vert f\left( x\right) g\left( x\right)
\right\vert dx$ can be estimated by the same way given in (i) above:%
\begin{equation*}
\int\nolimits_{E_{2}}\left\vert f\left( x\right) \right\vert dx\leq
c_{0}\left\Vert f\right\Vert _{\varphi ,W}
\end{equation*}%
(c) Using $W_{\ast }\in A_{p\left( \tilde{\varphi}\right) }$ we can proceed
as in (i) to obtain%
\begin{equation*}
\int\nolimits_{E_{1}}\left\vert g\left( x\right) \right\vert dx\leq
c_{1}\left\Vert f\right\Vert _{\tilde{\varphi},W_{\ast }}.
\end{equation*}%
(d) Mixing methods given in (b) and (c) we find

$\int\limits_{E_{3}}\left\vert \left( fg\right) \left( x\right) \right\vert
dx\leq \left( \int\limits_{E_{3}}\left\vert f\right\vert ^{p\left( \varphi
\right) }W\left( x\right) dx\right) ^{\frac{1}{p\left( \varphi \right) }%
}\left( \int\limits_{E_{3}}\left\vert g\right\vert ^{(p\left( \varphi
\right) )^{\prime }}W\left( x\right) ^{-\frac{(p\left( \varphi \right)
)^{\prime }}{p\left( \varphi \right) }}dx\right) ^{\frac{1}{(p\left( \varphi
\right) )^{\prime }}}\leq c_{2}$\qquad \qquad since $W\in A_{p\left( \varphi
\right) }$ and $W_{\ast }\in A_{p\left( \tilde{\varphi}\right) }\cap
A_{(p\left( \varphi \right) )^{\prime }}.$ From this we get%
\begin{equation*}
\int\nolimits_{E_{3}}\left\vert f\left( x\right) \right\vert \left\vert
g\left( x\right) \right\vert dx\leq c_{2}\left\Vert f\right\Vert _{\varphi
,W}\left\Vert g\right\Vert _{\tilde{\varphi},W_{\ast }}.
\end{equation*}%
Collecting results above we have%
\begin{equation*}
\int\nolimits_{\mathsf{T}}\left\vert f\left( x\right) \right\vert \left\vert
g\left( x\right) \right\vert dx\leq \mathbb{H}\left\Vert f\right\Vert
_{\varphi ,W}\left\Vert g\right\Vert _{\tilde{\varphi},W_{\ast }}
\end{equation*}%
with $\mathbb{H}$:$\mathbb{=(}2\pi +\sum_{i=0}^{2}c_{i}).$
\end{proof}

\begin{proof}[\textbf{Proof of Theorem \protect\ref{FuX}}]
Proof of Theorem \ref{FuX} is the same with proof of Theorem 1 of \cite%
{akgArx} and therefore omitted.
\end{proof}

\begin{proof}[\textbf{Proof of Theorem \protect\ref{traX}}]
We follow proof of Theorem 10 of \cite{akgArx}. Suppose that $\varphi \in Q$%
, $W\in A_{p\left( \varphi \right) }$ and $0\leq f$,$g\in L_{W}^{\varphi }$.
If $\left\Vert g\right\Vert _{\varphi ,W}=0$, then the result (\ref{rrr}) is
obvious. So, we assume that $\left\Vert g\right\Vert _{\varphi ,W}>0$. Set $%
C_{1}$:=$\left\Vert G\right\Vert _{\infty }$. Using hypotesis we get%
\begin{equation*}
\left\Vert \Xi _{g,G}\right\Vert _{C\left( \mathsf{T}\right) }\leq
C\left\Vert \Xi _{f,G}\right\Vert _{C\left( \mathsf{T}\right) }
\end{equation*}%
\begin{equation*}
=C\left\Vert \int\nolimits_{\mathsf{T}}f\left( x+u\right) \left\vert G\left(
x\right) \right\vert dx\right\Vert _{C\left( \mathsf{T}\right) }=C\max_{u\in 
\mathsf{T}}\int\nolimits_{\mathsf{T}}f\left( x+u\right) \left\vert G\left(
x\right) \right\vert dx
\end{equation*}%
\begin{equation*}
\leq C\max_{u\in \mathsf{T}}\left\Vert f\left( \cdot +u\right) \right\Vert
_{1}\left\Vert G\right\Vert _{\infty }=CC_{1}\left\Vert f\right\Vert
_{1}\leq CC_{1}c_{0}\left\Vert f\right\Vert _{\varphi ,W}.
\end{equation*}

Using norm conjugate formula%
\begin{equation*}
\frac{1}{c_{\ast }}\left\Vert g\right\Vert _{\varphi ,W}=\underset{\underset{%
\left\Vert G\right\Vert _{\tilde{\varphi},W_{\ast }}\leq 1}{G\in L_{W_{\ast
}}^{\tilde{\varphi}}\cap S\left( \mathsf{T}\right) }}{\sup }\int\nolimits_{%
\mathsf{T}}g\left( x\right) \left\vert G\left( x\right) \right\vert dx,
\end{equation*}%
(here constant $c_{\ast }$ comes from inequality%
\begin{equation*}
\overset{\thickapprox }{\varphi }\left( t\right) \leq \varphi \left(
t\right) \leq \overset{\thickapprox }{\varphi }\left( c_{\ast }t\right) 
\text{ )}
\end{equation*}%
for any $\varepsilon >0$ there exists a $G\in L_{W_{\ast }}^{\tilde{\varphi}}
$ with $\left\Vert G\right\Vert _{\tilde{\varphi},W_{\ast }}\leq 1$ such that%
\begin{equation*}
\int\nolimits_{\mathsf{T}}g\left( x\right) \left\vert G\left( x\right)
\right\vert dx\geq \frac{1}{c_{\ast }}\left\Vert g\right\Vert _{\varphi
,W}-\varepsilon 
\end{equation*}%
and hence%
\begin{eqnarray*}
\left\Vert \Xi _{g,G}\right\Vert _{C\left( \mathsf{T}\right) } &\geq
&\left\vert \Xi _{g}\left( 0\right) \right\vert \geq \frac{1}{c_{\ast }}%
\int\nolimits_{\mathsf{T}}g\left( x\right) \left\vert G\left( x\right)
\right\vert dx \\
&\geq &\frac{1}{c_{\ast }}\left\Vert g\right\Vert _{\varphi ,W}-\varepsilon .
\end{eqnarray*}%
Now taking limit $\varepsilon \rightarrow 0+$ we have%
\begin{equation*}
\left\Vert \Xi _{g,G}\right\Vert _{C\left( \mathsf{T}\right) }\geq \frac{1}{%
c_{\ast }}\left\Vert g\right\Vert _{\varphi ,W}
\end{equation*}%
and hence%
\begin{equation*}
\left\Vert g\right\Vert _{\varphi ,W}\leq c_{\ast }\left\Vert \Xi
_{g,G}\right\Vert _{C\left( \mathsf{T}\right) }\leq c_{\ast }C\left\Vert \Xi
_{f}\right\Vert _{C\left( \mathsf{T}\right) }\leq c_{\ast
}CC_{1}c_{0}\left\Vert f\right\Vert _{\varphi ,W}.
\end{equation*}%
For general case $f$,$g\in L_{W}^{\varphi }$we have%
\begin{equation*}
\left\Vert g\right\Vert _{\varphi ,W}\leq 2c_{\ast }CC_{1}c_{0}\left\Vert
f\right\Vert _{\varphi ,W}.
\end{equation*}
\end{proof}

\begin{proof}[\textbf{Proof of Lemma \protect\ref{l2x}}]
Using $\Xi _{\mathcal{A}_{h}f,G}=\mathcal{A}_{h}\Xi _{f,G}$ we get $\Xi
_{\left( \mathcal{A}_{h}\right) ^{i}f,G}=\left( \mathcal{A}_{h}\right)
^{i}\Xi _{f,G}$ and%
\begin{equation*}
\left\Vert \left( \mathcal{A}_{h}\right) ^{i}f\right\Vert _{\varphi ,W}\leq
2c_{\ast }\left\Vert \Xi _{\left( \mathcal{A}_{h}\right) ^{i}f,G}\right\Vert
_{C\left( \mathsf{T}\right) }=2c_{\ast }\left\Vert \left( \mathcal{A}%
_{h}\right) ^{i}\Xi _{f,G}\right\Vert _{C\left( \mathsf{T}\right) }
\end{equation*}%
\begin{equation*}
\leq 2c_{\ast }\left\Vert \Xi _{f,G}\right\Vert _{C\left( \mathsf{T}\right)
}\leq 2c_{\ast }C_{1}c_{0}\left\Vert f\right\Vert _{\varphi ,W}.
\end{equation*}%
Here constant $2c_{\ast }C_{1}c_{0}$ does not depend on $i$.
\end{proof}

\begin{proof}[\textbf{Proof of Lemma \protect\ref{l2xx}}]
From \cite[p.14, (1.51)]{skm} we know that%
\begin{equation*}
\left\vert \frac{\left( -1\right) ^{j}\Gamma \left( k+1\right) }{\Gamma
\left( j+1\right) \Gamma (k-j+1)}\right\vert \lesssim \frac{1}{j^{k+1}}\text{
\ for }j\in \mathbb{N}.
\end{equation*}%
By Lemma \ref{l2x} and (\ref{bina}) we obtain, for $N\in \mathbb{N}$,%
\begin{equation*}
\left\Vert \sum\limits_{j=0}^{N}\frac{\left( -1\right) ^{j}\Gamma \left(
k+1\right) }{\Gamma \left( j+1\right) \Gamma (k-j+1)}\left( \mathcal{A}%
_{h}\right) ^{j}f\right\Vert _{\varphi ,W}\leq \sum\limits_{j=0}^{N}\frac{%
c_{k}}{j^{k+1}}\left\Vert \left( \mathcal{A}_{h}\right) ^{j}f\right\Vert
_{\varphi ,W}
\end{equation*}%
\begin{equation*}
\leq \left\Vert f\right\Vert _{\varphi ,W}C_{\varphi
,W}c_{k}\sum\limits_{j=0}^{N}\frac{1}{j^{k+1}}\leq \left\Vert f\right\Vert
_{\varphi ,W}C_{\varphi ,W}c_{k}\sum\limits_{j=0}^{\infty }\frac{1}{j^{k+1}}%
\leq C_{k,\varphi ,W}\left\Vert f\right\Vert _{\varphi ,W}<\infty .
\end{equation*}%
Using monotone convergence theorem we obtain (\ref{AX}).
\end{proof}

\begin{proof}[\textbf{Proof of Lemma \protect\ref{y}}]
For any $h\in \mathbb{R}^{+}$, we have%
\begin{equation*}
\left( I-\mathcal{A}_{h}\right) ^{k}=\frac{1}{2}\left( I-\mathcal{A}%
_{h}\right) ^{k}+\frac{1}{2}\left( I-\mathcal{A}_{h}\right) ^{k+1}+\frac{1}{2%
}\mathcal{A}_{h}\left( I-\mathcal{A}_{h}\right) ^{k}
\end{equation*}%
\begin{equation*}
\leq \frac{1}{2}\left( I-\mathcal{A}_{h}\right) ^{k+1}+\frac{1}{2}\left( I+%
\mathcal{A}_{h}\right) \left( I-\mathcal{A}_{h}\right) ^{k}
\end{equation*}%
\begin{equation*}
=\frac{1}{2}\left( I-\mathcal{A}_{h}\right) ^{k+1}+\frac{1}{2}\left( I+%
\mathcal{A}_{h}\right) \left( \frac{1}{2}\left( I-\mathcal{A}_{h}\right)
^{k+1}+\frac{1}{2}\left( I+\mathcal{A}_{h}\right) \left( I-\mathcal{A}%
_{h}\right) ^{k}\right)
\end{equation*}%
\begin{equation*}
=\left( \frac{1}{2}+\frac{1}{2^{2}}\left( I+\mathcal{A}_{h}\right) \right)
\left( I-\mathcal{A}_{h}\right) ^{k+1}+\frac{1}{2^{2}}\left( I+\mathcal{A}%
_{h}\right) ^{2}\left( I-\mathcal{A}_{h}\right) ^{k}\cdots
\end{equation*}%
\begin{equation*}
=\sum\nolimits_{j=1}^{k}\frac{1}{2^{j}}\left( I+\mathcal{A}_{h}\right)
^{j-1}\left( I-\mathcal{A}_{h}\right) ^{k+1}+\frac{1}{2^{k}}\left( I-%
\mathcal{A}_{h}^{2}\right) ^{k}.
\end{equation*}%
Hence, if $f\in C\left( \mathsf{T}\right) $, then,%
\begin{equation*}
\Vert \left( I-\mathcal{A}_{h}\right) ^{k}f\Vert _{C\left( \mathsf{T}\right)
}\leq \frac{1}{2^{k}}\Vert \left( I-\mathcal{A}_{h}^{2}\right) ^{k}\Vert
_{C\left( \mathsf{T}\right) }+\left( \sum\nolimits_{j=1}^{k}\frac{1}{2^{j}}%
\right) \Vert \left( I-\mathcal{A}_{h}\right) ^{k+1}\Vert _{C\left( \mathsf{T%
}\right) }
\end{equation*}%
\begin{equation*}
\leq \frac{1}{2^{k}}\Vert \left( I-\mathcal{A}_{h}^{2}\right) ^{k}\Vert
_{C\left( \mathsf{T}\right) }+\Vert \left( I-\mathcal{A}_{h}\right)
^{k+1}\Vert _{C\left( \mathsf{T}\right) }
\end{equation*}%
\begin{equation*}
\leq \frac{1}{2^{2k}}\Vert \left( I-\mathcal{A}_{h}^{4}\right) ^{k}f\Vert
_{\varphi ,W}+2\Vert \left( I-\mathcal{A}_{h}\right) ^{k+1}\Vert _{C\left( 
\mathsf{T}\right) }\cdots
\end{equation*}%
\begin{equation*}
\leq \frac{1}{2^{mk}}\Vert \left( I-\mathcal{A}_{h}^{2^{m}}\right)
^{k}f\Vert _{C\left( \mathsf{T}\right) }+2^{m}\Vert \left( I-\mathcal{A}%
_{h}\right) ^{k+1}\Vert _{C\left( \mathsf{T}\right) }
\end{equation*}%
\begin{equation*}
\leq \frac{2^{k}}{2^{mk}}\Vert f\Vert _{C\left( \mathsf{T}\right)
}+2^{m}\Vert \left( I-\mathcal{A}_{h}\right) ^{k+1}\Vert _{C\left( \mathsf{T}%
\right) }.
\end{equation*}%
Using%
\begin{equation*}
\Xi _{\left( I-\mathcal{A}_{h}\right) ^{k}f,G}=\left( I-\mathcal{A}%
_{h}\right) ^{k}\Xi _{f,G}
\end{equation*}%
and Theorem \ref{traX} we obtain%
\begin{equation*}
\Vert \left( I-\mathcal{A}_{h}\right) ^{k}f\Vert _{\varphi ,W}\leq \frac{%
2^{k}\mathbf{c}_{2}}{2^{km}}\Vert f\Vert _{\varphi ,W}+\mathbf{c}%
_{2}2^{m}\Vert \left( I-\mathcal{A}_{h}\right) ^{k+1}f\Vert _{\varphi ,W}.
\end{equation*}
\end{proof}

\begin{proof}[\textbf{Proof of Theorem \protect\ref{teo2}}]
Let $n\in \mathbb{N}$ and $L_{W}^{\varphi }$ be fixed. In the case $k=1$ we
know from \cite{ai11} that%
\begin{equation}
E_{n}(f)_{\varphi ,W}\lesssim \Omega _{1}\left( f,\frac{1}{n}\right)
_{\varphi ,W}.  \notag
\end{equation}%
The case integer $k\geq 2$: Following the idea given in \cite{d06}, we will
use induction on $k$. We suppose that the inequality (\ref{JT}) holds for $%
g\in L_{W}^{\varphi }$ and some $k=2,3,4,...$:%
\begin{equation}
E_{n}(g)_{\varphi ,W}\lesssim \Omega _{k}\left( g,\frac{1}{n}\right)
_{\varphi ,W}.  \label{ind}
\end{equation}%
We have to verify the fulfilment of inequality (\ref{JT}) for $k+1$. We will
use the operator $S_{n}f$,\ the $n$th partial sum of the Fourier series (\ref%
{fo}) of $f$. We will obtain the estimate%
\begin{equation*}
\Vert f-S_{n}f\Vert _{\varphi ,W}\lesssim \Omega _{k+1}\left( f,\frac{1}{n}%
\right) _{\varphi ,W}.
\end{equation*}%
We set $u(\cdot ):=f(\cdot )-S_{n}f(\cdot )$. Hence we get $S_{n}(u)(\cdot )$%
=$S_{n}(f-S_{n}f)(\cdot )$=$S_{n}(f)(\cdot )-S_{n}\left( S_{n}f\right)
\left( \cdot \right) $=$0$. Since $f\rightarrow S_{n}f$ is uniformly bounded
in $L_{W}^{\varphi }$ (see (\ref{esen})), we have%
\begin{equation*}
\Vert f-S_{n}f\Vert _{\varphi ,W}\lesssim E_{n}(f)_{\varphi ,W},
\end{equation*}%
and using induction hypothesis (\ref{ind})%
\begin{equation*}
\Vert u\Vert _{\varphi ,W}=\Vert u-S_{n}(u)\Vert _{\varphi ,W}\lesssim
E_{n}(u)_{\varphi ,W}\leq \mathbf{C}_{1}\Omega _{k}\left( u,\frac{1}{n}%
\right) _{\varphi ,W}.
\end{equation*}%
We know, from Lemma \ref{y} that%
\begin{equation*}
\Omega _{k}\left( u,\frac{1}{n}\right) _{\varphi ,W}\leq C\delta ^{mk}\Vert
u\Vert _{\varphi ,W}+C^{\prime }\Omega _{k+1}\left( u,\frac{1}{n}\right)
_{\varphi ,W}.
\end{equation*}%
Choosing $m$ so big that $C$\textbf{$C_{1}$}$\delta ^{mk}<1/2$, we get 
\begin{equation*}
\Vert u\Vert _{\varphi ,W}\leq \mathbf{C_{1}}\Omega _{k}\left( u,\frac{1}{n}%
\right) _{\varphi ,W}\leq C\mathbf{C_{1}}\delta ^{mk}\Vert u\Vert _{\varphi
,W}+C\Omega _{k+1}\left( u,\frac{1}{n}\right) _{\varphi ,W}.
\end{equation*}%
Therefore%
\begin{equation*}
\Vert u\Vert _{\varphi ,W}\lesssim \Omega _{k+1}\left( u,\frac{1}{n}\right)
_{\varphi ,W}.
\end{equation*}%
From uniform boundedness of operator $f\longmapsto S_{n}f$ in $%
L_{W}^{\varphi }$ we have 
\begin{equation*}
\Omega _{k+1}\left( u,\frac{1}{n}\right) _{\varphi ,W}\lesssim \Omega
_{k+1}\left( f,\frac{1}{n}\right) _{\varphi ,W}
\end{equation*}%
and the result 
\begin{equation*}
E_{n}\left( f\right) _{\varphi ,W}\lesssim \Vert f-S_{n}f\Vert _{\varphi ,W}%
\text{=}\Vert u\Vert _{\varphi ,W}\lesssim \Omega _{k+1}\left( u,\frac{1}{n}%
\right) _{\varphi ,W}\lesssim \Omega _{k+1}\left( f,\frac{1}{n}\right)
_{\varphi ,W}
\end{equation*}%
holds for any $k\in \mathbb{N}$. The case $k\in \mathbb{R}^{+}\backslash 
\mathbb{N}$ can easily be obtained from the last inequality as, for example,
in \cite{Akgeja}. The proof is completed.
\end{proof}

\begin{definition}
A weight $\rho $ is in class $A_{1}\left( W\right) $ if $\rho \in A_{1}$
with respect to $W\left( x\right) dx.$
\end{definition}

The proof of Lemma \ref{A11} below can be proceed as in the Lemma 4.3 of 
\cite{cmp04}.

\begin{lemma}
\label{A11}If $W\in A_{1}$ and $\omega \in A_{1}\left( W\right) $ then $%
W\omega \in A_{1}.$
\end{lemma}

\begin{proof}[\textbf{Proof of Theorem \protect\ref{ET}}]
We can give a modification of the proof of Theorem 3.1 in \cite{CGMP}. Let $%
\psi (u):=\varphi \left( u^{1/p_{0}}\right) $ be a convex function, $W\in
A_{p\left( \varphi \right) }$ and%
\begin{equation*}
\int\limits_{\mathsf{T}}\psi \left( F\left( x\right) \right) W\left(
x\right) dx\text{,}\int\limits_{\mathsf{T}}\psi \left( g\left( x\right)
\right) W\left( x\right) dx<\infty .
\end{equation*}%
From $\varphi \in \bigtriangleup _{2}$ we have $\psi \in \bigtriangleup _{2}$
and hence $\tilde{\psi}^{\alpha }$ is quasiconvex for some $\alpha \in
\left( 0,1\right) $ (see Lemma 6.1.6 of \cite{GGKK}). By (ii) of Lemma 1.1.1
of \cite{KK91}, for $0<\theta <1$ and $t\geq 0$%
\begin{equation*}
\tilde{\psi}\left( \theta t\right) =\left( \tilde{\psi}^{\alpha }\left(
\theta t+(1-\theta )0\right) \right) ^{1/\alpha }\leq a_{1}^{1/\alpha
}\theta ^{1/\alpha }\tilde{\psi}\left( a_{1}t\right)
\end{equation*}%
where $a_{1}>1$. On the other hand following the same lines of the proof of
Proposition 5.1 in \cite{CGMP} we get%
\begin{equation*}
\int\limits_{\mathsf{T}}\tilde{\psi}\left( M_{W}F\left( x\right) \right)
W\left( x\right) dx\leq a_{2}\int\limits_{\mathsf{T}}\tilde{\psi}\left(
a_{2}F\left( x\right) \right) W\left( x\right) dx
\end{equation*}%
for some $a_{2}\geq 1$, where%
\begin{equation*}
M_{W}f\left( x\right) :=\sup_{B\ni x}\frac{1}{W\left( B\right) }%
\int\limits_{B}\left\vert f\left( y\right) \right\vert W\left( y\right) dy,%
\text{\quad }f\in L^{1},
\end{equation*}%
is the weighted Hardy-Littlewood maximal operator.

We set $a_{0}:=\max \left\{ a_{1}^{1/\alpha },a_{2}\right\} $. Then $%
a_{0}>1. $ From this%
\begin{eqnarray}
\int\limits_{\mathsf{T}}\tilde{\psi}\left( \frac{M_{W}F\left( x\right) }{%
a_{0}}\right) W\left( x\right) dx &\leq &a_{0}\int\limits_{\mathsf{T}}\tilde{%
\psi}\left( F\left( x\right) \right) W\left( x\right) dx,  \notag \\
\tilde{\psi}\left( \theta t\right) &\leq &a_{0}\theta ^{1/\alpha }\tilde{\psi%
}\left( a_{0}t\right) ,\quad \forall \theta \in \left( 0,1\right) .
\label{alti}
\end{eqnarray}%
By means of the Proposition 6.1.3 of \cite{GGKK}, there exists a number $%
\varepsilon _{0}\in \left( 0,1\right) $ such that%
\begin{equation}
\tilde{\psi}\left( \varepsilon _{0}\frac{\psi \left( t\right) }{t}\right)
\leq \psi \left( t\right) \text{,\quad }\forall t>0.  \label{bak}
\end{equation}

We suppose that%
\begin{equation*}
0\leq h\left( x\right) \text{:=}\frac{\theta \varepsilon _{0}\psi \left(
F^{p_{0}}\left( x\right) \right) }{a_{0}F^{p_{0}}\left( x\right) }\text{ for 
}F\left( x\right) >0,
\end{equation*}%
\begin{equation*}
h\left( x\right) \text{:=}0\text{ for }F\left( x\right) =0,
\end{equation*}%
where $\theta \in \left( 0,1\right) $ a fixed number that will be chosen
appropriately later. We define Rubio de Francia's algorithm as%
\begin{equation*}
Rh\left( x\right) =\frac{2a_{0}-1}{2a_{0}}\sum\limits_{k=0}^{\infty }\frac{1%
}{\left( 2a_{0}\right) ^{k}}\frac{M_{W}^{k}h\left( x\right) }{a_{0}^{k}}.
\end{equation*}%
We can obtain the properties

(\textbf{1}) $h\left( x\right) \leq \frac{2a_{0}}{2a_{0}-1}Rh\left( x\right)
.$

(\textbf{2}) $\int\limits_{\mathsf{T}}\tilde{\psi}\left( Rh\left( x\right)
\right) W\left( x\right) dx\leq \frac{2a_{0}-1}{2a_{0}}\int\limits_{\mathsf{T%
}}\tilde{\psi}\left( h\left( x\right) \right) W\left( x\right) dx.$

(\textbf{3}) $M_{W}Rh\left( x\right) \leq 2a_{0}^{2}Rh\left( x\right) .$

From (\textbf{3}) we get $Rh\in A_{1}\left( W\right) $ with some constant
independent of $f$. By (\textbf{1}) we obtain%
\begin{eqnarray*}
\int\limits_{\mathsf{T}}\varphi \left( F\left( x\right) \right) W\left(
x\right) dx &=&\int\limits_{\mathsf{T}}\psi \left( F^{p_{0}}\left( x\right)
\right) W\left( x\right) dx=\frac{a_{0}}{\theta }\int\limits_{\mathsf{T}%
}F^{p_{0}}\left( x\right) h\left( x\right) W\left( x\right) dx \\
&\leq &\frac{2a_{0}^{2}}{\left( 2a_{0}-1\right) \theta }\int\limits_{\mathsf{%
T}}F^{p_{0}}\left( x\right) Rh\left( x\right) W\left( x\right) dx.
\end{eqnarray*}%
By Young inequality (\ref{Young}) we find%
\begin{equation*}
\int\limits_{\mathsf{T}}F^{p_{0}}\left( x\right) Rh\left( x\right) Wdx\leq
\int\limits_{\mathsf{T}}\psi \left( F^{p_{0}}\left( x\right) \right) Wdx%
\text{+}\int\limits_{\mathsf{T}}\tilde{\psi}\left( Rh\left( x\right) \right)
Wdx\text{=:}I_{1}+I_{2}.
\end{equation*}%
By hypothesis $I_{1}<\infty $. Using $0<\theta <1$, (\ref{alti}), (\textbf{2}%
) and (\ref{bak}) we get%
\begin{eqnarray}
I_{2} &\leq &\frac{2a_{0}-1}{2a_{0}}\int\limits_{\mathsf{T}}\tilde{\psi}%
\left( h\left( x\right) \right) W\left( x\right) dx  \notag \\
&=&\frac{2a_{0}-1}{2a_{0}}\int\limits_{\mathsf{T}}\tilde{\psi}\left( \frac{%
\theta \varepsilon _{0}\psi \left( F^{p_{0}}\left( x\right) \right) }{%
a_{0}F^{p_{0}}\left( x\right) }\right) W\left( x\right) dx  \notag \\
&\leq &\frac{2a_{0}-1}{2a_{0}}a_{0}\theta ^{1/\alpha }\int\limits_{\mathsf{T}%
}\tilde{\psi}\left( \varepsilon _{0}\frac{\psi \left( F^{p_{0}}\left(
x\right) \right) }{F^{p_{0}}\left( x\right) }\right) W\left( x\right) dx 
\notag \\
&<&\left( 2a_{0}-1\right) \theta ^{1/\alpha }\int\limits_{\mathsf{T}}\psi
\left( F^{p_{0}}\left( x\right) \right) W\left( x\right) dx  \notag \\
&=&\left( 2a_{0}-1\right) \theta ^{1/\alpha }\int\limits_{\mathsf{T}}\varphi
\left( F\left( x\right) \right) W\left( x\right) dx<\infty .  \label{bunu}
\end{eqnarray}%
By hypothesis, Lemma \ref{A11} and (\ref{bunu})%
\begin{eqnarray*}
\int\limits_{\mathsf{T}}\varphi \left( F\left( x\right) \right) W\left(
x\right) dx &\leq &\frac{2a_{0}^{2}}{\left( 2a_{0}-1\right) \theta }%
\int\limits_{\mathsf{T}}F^{p_{0}}\left( x\right) Rh\left( x\right) W\left(
x\right) dx \\
&\leq &\frac{2a_{0}^{2}C}{\left( 2a_{0}-1\right) \theta }\int\limits_{%
\mathsf{T}}g^{p_{0}}\left( x\right) Rh\left( x\right) W\left( x\right) dx \\
&\leq &\frac{2a_{0}^{2}C\theta ^{-1}}{\left( 2a_{0}-1\right) }\left[
\int\limits_{\mathsf{T}}\psi \left( g^{p_{0}}\left( x\right) \right) W\left(
x\right) dx\text{+}\int\limits_{\mathsf{T}}\tilde{\psi}\left( Rh\left(
x\right) \right) W\left( x\right) dx\right]  \\
&\leq &\frac{2a_{0}^{2}C}{\left( 2a_{0}-1\right) \theta }\int\limits_{%
\mathsf{T}}\varphi \left( g\left( x\right) \right) W\left( x\right) dx+ \\
&&+2a_{0}^{2}\left( C+1\right) \theta ^{\frac{1-\alpha }{\alpha }%
}\int\limits_{\mathsf{T}}\varphi \left( F\left( x\right) \right) W\left(
x\right) dx.
\end{eqnarray*}%
Choosing $\theta :=\left( 4a_{0}^{2}\left( C+1\right) \right) ^{-\frac{%
\alpha }{1+\alpha }}\in \left( 0,1\right) $ and collecting above results%
\begin{equation*}
\int\limits_{\mathsf{T}}\varphi \left( F\left( x\right) \right) Wdx\leq 
\frac{2a_{0}^{2}C\theta ^{-1}}{\left( 2a_{0}-1\right) }\int\limits_{\mathsf{T%
}}\varphi \left( g\left( x\right) \right) Wdx\text{+}\frac{1}{2}\int\limits_{%
\mathsf{T}}\varphi \left( F\left( x\right) \right) Wdx.
\end{equation*}%
Thus%
\begin{equation*}
\int\limits_{\mathsf{T}}\varphi \left( F\left( x\right) \right) W\left(
x\right) dx\lesssim \int\limits_{\mathsf{T}}\varphi \left( g\left( x\right)
\right) W\left( x\right) dx.
\end{equation*}
\end{proof}

\begin{proof}[\textbf{Proof of Theorem \protect\ref{mmt}}]
The following weighted Marcinkiewicz multiplier theorem was proved in \cite%
{k80}. We will use its one dimensional 2$\pi $-periodic version: Let a
sequence $\left\{ \lambda _{\mu }\right\} $\ of real numbers be satisfy%
\begin{equation*}
\left\vert \lambda _{\mu }\right\vert \leq A\text{,\quad }\sum\limits_{\mu
=2^{m-1}}^{2^{m}-1}\left\vert \lambda _{\mu }-\lambda _{\mu +1}\right\vert
\leq A
\end{equation*}%
for all $\mu ,m\in \mathbb{N}$. If $1<p<\infty $, $W\in A_{p}$\ and $f\in
L_{W}^{p}$\ with the Fourier series (\ref{fo}), then there is a function $%
G\in L_{W}^{p}$\ such that the series $\sum\nolimits_{k=-\infty }^{\infty
}\lambda _{k}A_{k}\left( \cdot ,f\right) $\ is Fourier series for $G$\ and%
\begin{equation}
\left\Vert G\right\Vert _{p,W}\leq C\left\Vert f\right\Vert _{p,W}
\label{MM1}
\end{equation}%
where $C$\ does not depend on $f$.

Using (1) and (3) of Lemma 6.1.6 of \cite[p.215]{GGKK}; (1) and (3) of Lemma
6.1.1 of \cite[p.211]{GGKK} and $\overset{\thickapprox }{\varphi }\leq
\varphi $ we have $\overset{\thickapprox }{\varphi }\sim \varphi $. Then $%
\overset{\thickapprox }{\varphi }\in \bigtriangleup _{2}$ because $\varphi
\in \bigtriangleup _{2}$. On the other hand since $\varphi ^{\theta }$ is
quasiconvex for some $\theta \in \left( 0,1\right) $, by Lemma 6.1.6 of \cite%
[p.215]{GGKK} we have that $\varphi $ is quasiconvex, which implies that $%
\tilde{\varphi}$ is also quasiconvex. This property together with the
relation $\overset{\thickapprox }{\varphi }\in \bigtriangleup _{2}$ is
equivalent to the quasiconvexity of $\tilde{\varphi}^{\beta }$ for some $%
\beta \in \left( 0,1\right) $, by Lemma 6.1.6 of \cite[p.215]{GGKK}.
Therefore by definition we have $\left( p\left( \varphi \right) \right)
^{\prime }$, $\left( p\left( \tilde{\varphi}\right) \right) ^{\prime
}<\infty $. Hence, we can choose the numbers $r,s$ such that $\left( p\left( 
\tilde{\varphi}\right) \right) ^{\prime }<s<\infty $, $r<p\left( \varphi
\right) $ and $W\in A_{r}$. On the other hand, from (\ref{MM1}) the
boundedness of the linear operator%
\begin{equation*}
f\rightarrow Uf\left( \cdot \right) :=\sum\nolimits_{k=0}^{\infty }\lambda
_{k}A_{k}\left( \cdot ,f\right)
\end{equation*}%
hold in $L_{W}^{p}$, in case of $W\in A_{p}$ $\left( 1<p<\infty \right) $,
where the series $\sum\nolimits_{k=0}^{\infty }\lambda _{k}A_{k}\left( \cdot
,f\right) $\ is the Fourier series of $G$. This implies that the operator $U$
is of weak types $\left( r,r\right) $ and $\left( s,s\right) $. Hence,
choosing $Y_{0}=Y_{1}=\mathsf{T}$, $S_{0}=S_{1}=B$ ($B$ is Borel $\sigma $%
-algebra), $d\upsilon _{0}=d\upsilon _{1}=W\left( x\right) dx$ and applying
Theorem \ref{l1} we have%
\begin{equation}
\int\limits_{\mathsf{T}}\varphi \left( \left\vert Uf\left( x\right)
\right\vert \right) W\left( x\right) dx\leq c\int\limits_{\mathsf{T}}\varphi
\left( \left\vert f\left( x\right) \right\vert \right) W\left( x\right) dx.
\label{**}
\end{equation}

Since $\varphi $ is quasiconvex, we have%
\begin{equation*}
\varphi \left( \alpha x\right) \leq \Phi \left( \alpha Cx\right) \leq \alpha
\Phi \left( Cx\right) \leq \alpha \varphi \left( Cx\right) \text{,\quad }%
\alpha \in \left( 0,1\right)
\end{equation*}%
for some convex Young function $\Phi $ and constant $C\geq 1$. Using this
inequality in (\ref{**}) for $f:=f/\lambda $, $\lambda >0$, we obtain the
inequality%
\begin{equation*}
\int\limits_{\mathsf{T}}\varphi \left( \frac{\left\vert U\left( \frac{f}{cC}%
\right) \left( x\right) \right\vert }{\lambda }\right) W\left( x\right)
dx\leq \int\limits_{\mathsf{T}}\varphi \left( \frac{\left\vert f\left(
x\right) \right\vert }{\lambda }\right) W\left( x\right) dx\text{,}
\end{equation*}%
which implies that%
\begin{equation*}
\left\Vert Uf\right\Vert _{\left( \varphi \right) ,W}\lesssim \left\Vert
f\right\Vert _{\left( \varphi \right) ,W}\text{.}
\end{equation*}%
The last relation is equivalent to the inequality%
\begin{equation*}
\left\Vert Uf\right\Vert _{\varphi ,W}\lesssim \left\Vert f\right\Vert
_{\varphi ,W}.
\end{equation*}%
and hence%
\begin{equation*}
\left\Vert G\right\Vert _{\varphi ,W}\lesssim \left\Vert f\right\Vert
_{\varphi ,W}
\end{equation*}%
holds with a positive constant, does not depend on $f$.
\end{proof}

\begin{proof}[\textbf{Proof of Theorem \protect\ref{lpt}}]
The following weighted Littlewood Paley theorem was proved in \cite{k80}. We
will use its one dimensional 2$\pi $-periodic version: if $1<p<\infty $, $%
W\in A_{p}$ and $f\in L_{W}^{p}$, then there are positive constants
depending only on $p,W$ such that%
\begin{equation*}
\left\Vert \left( \sum\limits_{l=0}^{\infty }\left\vert \triangledown
_{l}\right\vert ^{2}\right) ^{1/2}\right\Vert _{p,W}\approx \left\Vert
f\right\Vert _{p,W}.
\end{equation*}%
Taking in extrapolation Theorem \ref{ET} as $F$:=$\left(
\sum\limits_{l=0}^{\infty }\left\vert \triangledown _{l}\right\vert
^{2}\right) ^{1/2}$ and $g$:=$\left\vert f\right\vert $ we obtain%
\begin{equation*}
\left\Vert \left( \sum\limits_{l=0}^{\infty }\left\vert \triangledown
_{l}\right\vert ^{2}\right) ^{1/2}\right\Vert _{\varphi ,W}\lesssim
\left\Vert f\right\Vert _{\varphi ,W}.
\end{equation*}%
Reverse of the last inequality can be obtained by changing the role of the
functions $F$ and $g$.
\end{proof}

\begin{proof}[\textbf{Proof of Theorem \protect\ref{proc}}]
Let $r\in \mathbb{R}^{+}$, $\beta \in \left( 1,\infty \right) $, $n\in 
\mathbb{N}$ and we suppose that the number $m\in \mathbb{N}$ satisfies $%
2^{m}\leq n<2^{m+1}$. Using $E_{n}\left( f\right) _{\varphi ,W}\downarrow 0$
and Littlewood-Paley inequality (\ref{L-P}) we have%
\begin{eqnarray*}
J_{n,r}^{\beta } &\text{:}&\text{=}\frac{1}{n^{2r}}\left\{ \sum\limits_{\nu
=1}^{n}\nu ^{2\beta r-1}E_{\nu }^{\beta }\left( f\right) _{\varphi
,W}\right\} ^{1/\beta }\leq \frac{1}{n^{2r}}\left\{ \sum\limits_{\nu
=1}^{m+1}\sum\limits_{l=2^{\nu -1}}^{2^{\nu }-1}l^{2\beta r-1}E_{l}^{\beta
}\left( f\right) _{\varphi ,W}\right\} ^{1/\beta } \\
&\leq &\frac{1}{n^{2r}}\left\{ \sum\limits_{\nu =1}^{m+1}2^{2\nu \beta
r}E_{2^{\nu -1}-1}^{\beta }\left( f\right) _{\varphi ,W}\right\} ^{1/\beta }
\\
&\lesssim &\frac{1}{n^{2r}}\left\{ \sum\limits_{\nu =1}^{m+1}2^{2\nu \beta
r}\left\Vert \sum\limits_{l=2^{\nu -1}}^{\infty }A_{l}\left( x,f\right)
\right\Vert _{\varphi ,W}^{\beta }\right\} ^{1/\beta } \\
&\lesssim &\frac{1}{n^{2r}}\left\{ \sum\limits_{\nu =1}^{m+1}2^{2\nu \beta
r}\left\Vert \left( \sum\limits_{l=\nu }^{\infty }\left\vert \triangledown
_{l}\right\vert ^{2}\right) ^{\frac{1}{2}}\right\Vert _{\varphi ,W}^{\beta
}\right\} ^{\frac{1}{\beta }} \\
&\lesssim &\left\{ \sum\limits_{\nu =1}^{m+1}\left\Vert \left( \frac{2^{4\nu
r}}{n^{4r}}\sum\limits_{l=\nu }^{\infty }\left\vert \triangledown
_{l}\right\vert ^{2}\right) ^{\frac{1}{2}}\right\Vert _{\varphi ,W}^{\beta
}\right\} ^{1/\beta }\text{.}
\end{eqnarray*}%
We assume $\beta =2$. Then $2\geq q$ and%
\begin{equation*}
J_{n,r}^{2}\lesssim \left\{ \sum\limits_{\nu =1}^{m+1}\left\Vert \left( 
\frac{2^{4\nu r}}{n^{4r}}\sum\limits_{\mu =\nu }^{\infty }\left\vert
\triangledown _{l}\right\vert ^{2}\right) ^{1/2}\right\Vert _{\varphi
,W}^{2}\right\} ^{1/2}\text{.}
\end{equation*}%
Since the $l_{p}$ norm is decrease with $p\uparrow $ we have%
\begin{equation*}
J_{n,r}^{2}\lesssim \left\{ \sum\limits_{\nu =1}^{m+1}\left\Vert \left( 
\frac{2^{4\nu r}}{n^{4r}}\sum\limits_{l=\nu }^{\infty }\left\vert
\triangledown _{l}\right\vert ^{2}\right) ^{1/2}\right\Vert _{\varphi
,W}^{q}\right\} ^{1/q}\text{.}
\end{equation*}%
Using $q$ concavity of $L_{W}^{\varphi }$ we obtain%
\begin{equation*}
J_{n,r}^{2}\lesssim \left\Vert \left( \sum\limits_{\nu =1}^{m+1}\left( \frac{%
2^{4\nu r}}{n^{4r}}\sum\limits_{l=\nu }^{\infty }\left\vert \triangledown
_{l}\right\vert ^{2}\right) ^{q/2}\right) ^{1/q}\right\Vert _{\varphi
,W}\lesssim \left\Vert \sum\limits_{\nu =1}^{m+1}\frac{2^{2\nu r}}{n^{2r}}%
\sum\limits_{l=\nu }^{\infty }\left\vert \triangledown _{l}\right\vert
\right\Vert _{\varphi ,W}\text{.}
\end{equation*}%
Abel's transformation and Minkowski's inequality imply that%
\begin{eqnarray*}
J_{n,r}^{2} &\lesssim &\left\Vert \sum\limits_{\nu =1}^{m}\frac{2^{2\nu r}}{%
n^{2r}}\left\vert \triangledown _{\nu }\right\vert +\frac{2^{2r\left(
m+1\right) }}{n^{2r}}\sum\limits_{l=m+1}^{\infty }\left\vert \triangledown
_{l}\right\vert \right\Vert _{\varphi ,W} \\
&\lesssim &\left\Vert \sum\limits_{\nu =1}^{m}\frac{2^{2\nu r}}{n^{2r}}%
\left\vert \triangledown _{\nu }\right\vert \right\Vert _{\varphi
,W}+\left\Vert \frac{2^{2r\left( m+1\right) }}{n^{2r}}\sum\limits_{l=m+1}^{%
\infty }\left\vert \triangledown _{l}\right\vert \right\Vert _{\varphi ,W}%
\text{.}
\end{eqnarray*}%
If $\beta =q$, then $2\leq q$ and $q$ concavity of $L_{W}^{\varphi }$ imply
that%
\begin{eqnarray*}
J_{n,r}^{q} &\lesssim &\left( \sum\limits_{\nu =1}^{m+1}\left\Vert \left( 
\frac{2^{4\nu r}}{n^{4r}}\sum\limits_{l=\nu }^{\infty }\left\vert
\triangledown _{l}\right\vert ^{2}\right) ^{1/2}\right\Vert _{\varphi
,W}^{q}\right) ^{\frac{1}{q}} \\
&\lesssim &\left( \left\Vert \left( \sum\limits_{\nu =1}^{m+1}\left( \frac{%
2^{4\nu r}}{n^{4r}}\sum\limits_{l=\nu }^{\infty }\left\vert \triangledown
_{l}\right\vert ^{2}\right) ^{\frac{q}{2}}\right) ^{\frac{1}{q}}\right\Vert
_{\varphi ,W}\right)  \\
&\lesssim &\left\Vert \left( \sum\limits_{\nu =1}^{m+1}\frac{2^{4\nu r}}{%
n^{4r}}\sum\limits_{l=\nu }^{\infty }\left\vert \triangledown
_{l}\right\vert ^{2}\right) ^{1/2}\right\Vert _{\varphi ,W}\text{.}
\end{eqnarray*}%
Using Abel's transformation and Minkowski's inequality we get%
\begin{eqnarray*}
J_{n,r}^{q} &\lesssim &\left\Vert \left( \sum\limits_{\nu =1}^{m}\frac{%
2^{4\nu r}}{n^{4r}}\left\vert \triangledown _{\nu }\right\vert ^{2}+\frac{%
2^{4r\left( m+1\right) }}{n^{4r}}\sum\limits_{l=m+1}^{\infty }\left\vert
\triangledown _{l}\right\vert ^{2}\right) ^{1/2}\right\Vert _{\varphi ,W} \\
&\lesssim &\left\Vert \left( \sum\limits_{\nu =1}^{m}\frac{2^{4\nu r}}{n^{4r}%
}\left\vert \triangledown _{\nu }\right\vert ^{2}\right) ^{1/2}\right\Vert
_{\varphi ,W}+\left\Vert \left( \frac{2^{4r\left( m+1\right) }}{n^{4r}}%
\sum\limits_{l=m+1}^{\infty }\left\vert \triangledown _{l}\right\vert
^{2}\right) ^{1/2}\right\Vert _{\varphi ,W} \\
&\lesssim &\left\Vert \sum\limits_{\nu =1}^{m}\frac{2^{2\nu r}}{n^{2r}}%
\left\vert \triangledown _{\nu }\right\vert \right\Vert _{\varphi
,W}+\left\Vert \sum\limits_{l=m+1}^{\infty }\left\vert \triangledown
_{l}\right\vert \right\Vert _{\varphi ,W}\text{.}
\end{eqnarray*}%
Therefore we obtain%
\begin{equation*}
J_{n,r}^{\beta }\lesssim \left\Vert \sum\limits_{\nu =1}^{m}\frac{2^{2\nu r}%
}{n^{2r}}\left\vert \Delta _{\nu }\right\vert \right\Vert _{\varphi
,W}+\left\Vert \sum\limits_{l=m+1}^{\infty }\left\vert \Delta
_{l}\right\vert \right\Vert _{\varphi ,W}\text{.}
\end{equation*}%
Then%
\begin{equation*}
J_{n,r}^{\beta }\lesssim \left\Vert \sum\limits_{\nu
=1}^{m}\sum\limits_{l=2^{\nu -1}}^{2^{\nu }-1}\frac{2^{2\nu r}}{n^{2r}}%
A_{l}\left( x,f\right) \right\Vert _{\varphi ,W}+\left\Vert
\sum\limits_{l=2^{m}}^{\infty }A_{l}\left( x,f\right) \right\Vert _{\varphi
,W}\text{.}
\end{equation*}%
Since%
\begin{equation*}
\left\Vert f-S_{n}\left( f\right) \right\Vert _{\varphi ,W}\lesssim
E_{n}\left( f\right) _{\varphi ,W}
\end{equation*}%
we have%
\begin{equation*}
J_{n,r}^{\beta }\lesssim \left\Vert \sum\limits_{\nu
=1}^{m}\sum\limits_{l=2^{\nu -1}}^{2^{\nu }-1}\frac{l^{-2r}2^{2r\nu }\left(
l/2n\right) ^{2r}}{\left( 1-\frac{\sin l/2n}{l/2n}\right) ^{r}}\left( 1-%
\frac{\sin l/2n}{l/2n}\right) ^{r}A_{l}\left( \cdot ,f\right) \right\Vert
_{\varphi ,W}+E_{2^{m}-1}\left( f\right) _{\varphi ,W}
\end{equation*}%
and using Theorem \ref{teo2}%
\begin{equation*}
J_{n,r}^{\beta }\lesssim \left\Vert \sum\limits_{l=1}^{2^{m}-1}\frac{%
l^{-2r}2^{2r\nu }\left( l/2n\right) ^{2r}}{\left( 1-\frac{\sin l/2n}{l/2n}%
\right) ^{r}}\left( 1-\frac{\sin l/2n}{l/2n}\right) ^{r}A_{l}\left( \cdot
,f\right) \right\Vert _{\varphi ,W}+\Omega _{r}\left( f,\frac{1}{n}\right)
_{\varphi ,W}\text{.}
\end{equation*}%
We define%
\begin{equation*}
h_{l}:=\left\{ 
\begin{tabular}{ll}
$\frac{2^{2\nu r}}{l^{2r}}$ & , for $1\leq l\leq 2^{m}-1$, $\nu =1,\ldots ,m$%
, \\ 
$\frac{2^{2mr}}{l^{2r}}$ & , for $2^{m}\leq l\leq n$, \\ 
$0$ & , for $l>n$.%
\end{tabular}%
\right. 
\end{equation*}%
and%
\begin{equation*}
\lambda _{l}:=\left\{ 
\begin{tabular}{ll}
$\frac{\left( l/2n\right) ^{2r}}{\left( 1-\frac{\sin l/2n}{l/2n}\right) ^{r}}
$ & , for $1\leq l\leq 2n$, \\ 
$0$ & , for $l>2n$.%
\end{tabular}%
\right. 
\end{equation*}%
Hence, for $l=1,2,3,\ldots $, $\left\{ h_{l}\right\} $ and $\left\{ \lambda
_{l}\right\} $ satisfy (\ref{MS}). Taking%
\begin{equation*}
I:=\left\Vert \sum\limits_{\mu =1}^{2^{m}-1}\frac{2^{2\nu r}}{l^{2r}}\frac{%
\left( l/2n\right) ^{2r}}{\left( 1-\frac{\sin l/2n}{l/2n}\right) ^{r}}\left(
1-\frac{\sin l/2n}{l/2n}\right) ^{r}A_{l}\left( \cdot ,f\right) \right\Vert
_{\varphi ,W}
\end{equation*}%
we get%
\begin{equation*}
I=\left\Vert \sum\limits_{l=1}^{\infty }h_{l}\lambda _{l}\left( 1-\frac{\sin
l/2n}{l/2n}\right) ^{r}A_{l}\left( \cdot ,f\right) \right\Vert _{\varphi ,W}
\end{equation*}%
and using Theorem \ref{mmt} twice there hold%
\begin{equation*}
J_{n,r}^{\beta }\lesssim \left\Vert \sum\limits_{l=1}^{\infty }\left( 1-%
\frac{\sin l/2n}{l/2n}\right) ^{r}A_{l}\left( f\right) \right\Vert _{\varphi
,W}\text{=}\left\Vert \left( I-\mathcal{A}_{\frac{1}{n}}\right)
^{r}f\right\Vert _{\varphi ,W}\lesssim \Omega _{r}\left( f,\frac{1}{n}%
\right) _{\varphi ,W}\text{.}
\end{equation*}
\end{proof}

\begin{proof}[\textbf{Proof of Theorem \protect\ref{rea}}]
We prove (\ref{real}). Let $T_{n}$ be the near best approximating
trigonometric polynomial to $f.$ From Theorem \ref{teo2}%
\begin{equation*}
\left\Vert f-T_{n}\right\Vert _{\varphi ,W}\lesssim E_{n}\left( f\right)
_{\varphi ,W}\lesssim \Omega _{k}\left( f,\frac{1}{n}\right) _{\varphi ,W}.
\end{equation*}%
Thus using Lemma \ref{lem3}%
\begin{eqnarray*}
\frac{1}{n^{2k}}\left\Vert T_{n}^{(2k)}\right\Vert _{\varphi ,W} &\lesssim
&\Omega _{k}\left( T_{n},1/n\right) _{\varphi ,W} \\
&\lesssim &\Omega _{k}\left( T_{n}-f,1/n\right) _{\varphi ,W}+\Omega
_{k}\left( f,1/n\right) _{\varphi ,W} \\
&\lesssim &\left\Vert f-T_{n}\right\Vert _{\varphi ,W}+\Omega _{k}\left(
f,1/n\right) _{\varphi ,W}\lesssim \Omega _{k}\left( f,1/n\right) _{\varphi
,W}
\end{eqnarray*}%
and 
\begin{equation*}
\left\Vert f-T_{n}\right\Vert _{\varphi ,W}+\frac{1}{n^{2k}}\left\Vert
T_{n}^{(2k)}\right\Vert _{\varphi ,W}\lesssim \Omega _{k}\left( f,1/n\right)
_{\varphi ,W}.
\end{equation*}%
On the other hand from Lemma \ref{lem2}%
\begin{eqnarray*}
\Omega _{k}\left( f,1/n\right) _{\varphi ,W} &\leq &\Omega _{k}\left(
f-T_{n},1/n\right) _{\varphi ,W}+\Omega _{k}\left( T_{n},1/n\right)
_{\varphi ,W} \\
&\lesssim &\left\Vert f-T_{n}\right\Vert _{\varphi ,W}+\frac{1}{n^{2k}}%
\left\Vert T_{n}^{(2k)}\right\Vert _{\varphi ,W}=R_{2k}\left( f,1/n\right) .
\end{eqnarray*}%
Thus, (\ref{real}) is proved. The proof of the equivalence (\ref{Kf})
follows from properties of modulus of smoothness, $K$-functional and Lemma %
\ref{kubu}.
\end{proof}

\begin{proof}[\textbf{Proof of Theorem \protect\ref{teo4}}]
We use the method of Natanson and Timan, given in \cite{TiNa}. By Corollary %
\ref{co} we have, for $v\leq n,$%
\begin{equation*}
\Omega _{k}\left( f,1/v\right) _{\varphi ,W}\leq \left( 1+n/v\right) \Omega
_{k}\left( f,1/n\right) _{\varphi ,W}
\end{equation*}%
and hence%
\begin{equation*}
\prod\limits_{v=1}^{n}\Omega _{k}\left( f,1/v\right) _{\varphi ,W}\leq
\prod\limits_{v=1}^{n}\left( 1+\frac{n}{v}\right) \left( \Omega _{k}\left(
f,1/n\right) _{\varphi ,W}\right) ^{n}.
\end{equation*}%
For every $n\in \mathbb{N}$, we have%
\begin{equation*}
\prod\limits_{v=1}^{n}\left( 1+n/v\right) \leq \frac{2n}{\sqrt[n]{n!}}.
\end{equation*}%
Using Stirling formula%
\begin{equation*}
n!\approx \sqrt{2\pi n}n^{n}e^{-n}e^{\theta \left( n\right) }\text{ with }%
\left\vert \theta \left( n\right) \right\vert \leq 1/\left( 12n\right)
\end{equation*}%
we get%
\begin{equation*}
\prod\limits_{v=1}^{n}\left( 1+n/v\right) \leq 2^{k}e^{2k}.
\end{equation*}%
Thus%
\begin{equation*}
\left( \prod\limits_{v=1}^{n}\Omega _{k}\left( f,1/v\right) _{\varphi
,W}\right) ^{1/n}\lesssim \Omega _{k}\left( f,1/n\right) _{\varphi ,W}.
\end{equation*}%
From Theorem \ref{teo2} and the property $E_{n}\left( f\right) _{\varphi
,W}\downarrow $ as $n\uparrow $, we find%
\begin{equation*}
\left( \prod\limits_{v=1}^{n}E_{v}\left( f\right) _{\varphi ,W}\right)
^{1/n}\lesssim \left( \prod\limits_{v=1}^{n}\Omega _{k}\left( f,1/v\right)
_{\varphi ,W}\right) ^{1/n}\lesssim \Omega _{k}\left( f,1/n\right) _{\varphi
,W}
\end{equation*}%
and the result (\ref{IJ}) follows.
\end{proof}

\begin{proof}[\textbf{Proof of Lemma \protect\ref{lem2}}]
For $0<t\leq \frac{1}{n}$ we have%
\begin{eqnarray*}
\left\Vert (I-\mathcal{A}_{t})^{k}T_{n}\right\Vert _{\varphi ,W}
&=&\left\Vert \sum\limits_{j=0}^{n}\left( 1-\frac{\sin jt/2}{jt/2}\right)
^{k}A_{j}\left( \cdot ,T_{n}\right) \right\Vert _{\varphi ,W} \\
&=&\left\Vert \sum\limits_{j=1}^{n}\left( \frac{1-\frac{\sin jt/2}{jt/2}}{%
\left( jt/2\right) ^{2}}\right) ^{k}\left( jt/2\right) ^{2k}A_{j}\left(
\cdot ,T_{n}\right) \right\Vert _{\varphi ,W} \\
&\lesssim &\frac{1}{n^{2k}}\left\Vert \sum\limits_{j=1}^{n}\left( \frac{1-%
\frac{\sin jt/2}{jt/2}}{\left( jt/2\right) ^{2}}\right)
^{k}j^{2k}A_{j}\left( \cdot ,T_{n}\right) \right\Vert _{\varphi ,W}.
\end{eqnarray*}%
Using%
\begin{equation*}
\left( 1-\frac{\sin x}{x}\right) \leq x^{2}\text{,\quad }x\in \mathbb{R}%
^{+}\cup \left\{ 0\right\} ,
\end{equation*}%
and Marcinkiewicz type multiplier Theorem \ref{mmt} we have%
\begin{equation*}
\left\Vert (I-\mathcal{A}_{t})^{k}T_{n}\right\Vert _{\varphi ,W}\lesssim
n^{-2k}\left\Vert \sum\limits_{j=1}^{n}j^{2k}A_{j}\left( \cdot ,T_{n}\right)
\right\Vert _{\varphi ,W}.
\end{equation*}%
For $j\in \mathbb{N}$ we observe%
\begin{eqnarray*}
A_{j}\left( \cdot ,T_{n}\right) &=&a_{j}\cos j\left( \cdot +\frac{k\pi }{j}-%
\frac{k\pi }{j}\right) +b_{j}\sin j\left( \cdot +\frac{k\pi }{j}-\frac{k\pi 
}{j}\right) \\
&=&\cos k\pi \left[ a_{j}\cos j\left( \cdot +\frac{k\pi }{j}\right)
+b_{j}\sin j\left( \cdot +\frac{k\pi }{j}\right) \right] + \\
&&+\sin k\pi \left[ a_{j}\sin j\left( \cdot +\frac{k\pi }{j}\right)
-b_{j}\cos j\left( \cdot +\frac{k\pi }{j}\right) \right] \\
&=&A_{j}\left( \cdot +\frac{k\pi }{j},T_{n}\right) \cos k\pi +A_{j}\left(
\cdot +\frac{k\pi }{j},\widetilde{T_{n}}\right) \sin k\pi
\end{eqnarray*}%
and hence%
\begin{equation*}
\left\Vert (I-\mathcal{A}_{t})^{k}T_{n}\right\Vert _{\varphi ,W}\lesssim
\end{equation*}%
\begin{eqnarray*}
&\lesssim &\frac{1}{n^{2k}}\left\Vert \sum\limits_{j=1}^{n}j^{2k}\left[
A_{j}\left( \cdot +\frac{k\pi }{j},T_{n}\right) \cos k\pi +A_{j}\left( \cdot
+\frac{k\pi }{j},\widetilde{T_{n}}\right) \sin k\pi \right] \right\Vert
_{\varphi ,W} \\
&\lesssim &n^{-2k}\left( \left\Vert \sum\limits_{j=1}^{n}j^{2k}A_{j}\left(
\cdot +\frac{k\pi }{j},T_{n}\right) \right\Vert _{\varphi ,W}+\left\Vert
\sum\limits_{j=1}^{n}j^{2k}A_{j}\left( \cdot +\frac{k\pi }{j},\widetilde{%
T_{n}}\right) \right\Vert _{\varphi ,W}\right) .
\end{eqnarray*}

Since 
\begin{equation*}
A_{j}\left( \cdot ,T_{n}^{\left( 2k\right) }\right) =j^{2k}A_{j}\left( \cdot
+\frac{k\pi }{j},T_{n}\right) \text{,\quad }j\in \mathbb{N}
\end{equation*}%
using (\ref{esen})\ we obtain

\begin{equation*}
\Omega _{k}\left( T_{n},\frac{1}{n}\right) _{\varphi ,W}=\underset{0\leq
t\leq 1/n}{\sup }\left\Vert (I-\mathcal{A}_{t})^{k}T_{n}\right\Vert
_{\varphi ,W}
\end{equation*}%
\begin{equation*}
\lesssim n^{-2k}\left( \left\Vert \sum\limits_{j=1}^{n}j^{2k}A_{j}\left(
\cdot +\frac{k\pi }{j},T_{n}\right) \right\Vert _{\varphi ,W}+\left\Vert
\sum\limits_{j=1}^{n}j^{2k}A_{j}\left( \cdot +\frac{k\pi }{j},\widetilde{%
T_{n}}\right) \right\Vert _{\varphi ,W}\right)
\end{equation*}%
\begin{equation*}
\lesssim n^{-2k}\left( \left\Vert T_{n}^{\left( 2k\right) }\right\Vert
_{\varphi ,W}+\left\Vert \widetilde{T_{n}^{\left( 2k\right) }}\right\Vert
_{\varphi ,W}\right) \lesssim n^{-2k}\left\Vert T_{n}^{\left( 2k\right)
}\right\Vert _{\varphi ,W}
\end{equation*}%
and the proof is completed.
\end{proof}

\begin{proof}[\textbf{Proof of Lemma \protect\ref{lem3}}]
Let $T_{n}\left( x\right) =\sum\nolimits_{j=0}^{n}A_{j}(x,T_{n})$. Then,%
\begin{eqnarray*}
n^{-2k}\left\Vert T_{n}^{\left( 2k\right) }\right\Vert _{\varphi ,W}
&=&n^{-2k}\left\Vert \sum\limits_{j=1}^{n}j^{2k}A_{j}\left( \cdot +\frac{%
k\pi }{j},T_{n}\right) \right\Vert _{\varphi ,W} \\
&=&n^{-2k}\left\Vert \sum\limits_{j=1}^{n}j^{2k}\left( \cos k\pi A_{j}\left(
\cdot ,T_{n}\right) -\sin k\pi A_{j}\left( \cdot ,\widetilde{T_{n}}\right)
\right) \right\Vert _{\varphi ,W} \\
&\lesssim &\left\Vert \sum\limits_{j=1}^{n}\left( \frac{j}{2n}\right)
^{2k}A_{j}\left( \cdot ,T_{n}\right) \right\Vert _{\varphi ,W} \\
&&+\left\Vert \sum\limits_{j=1}^{n}\left( j/2n\right) ^{2k}A_{j}\left( \cdot
,\widetilde{T_{n}}\right) \right\Vert _{\varphi ,W} \\
&\lesssim &\left\Vert \sum\limits_{j=1}^{n}\left( \frac{\left( j/2n\right)
^{2}}{\left( 1-\frac{\sin j/2n}{j/2n}\right) }\right) ^{k}\left( 1-\frac{%
\sin j/2n}{j/2n}\right) ^{k}A_{j}\left( \cdot ,T_{n}\right) \right\Vert
_{\varphi ,W}+ \\
&&+\left\Vert \sum\limits_{j=1}^{n}\left( \frac{\left( j/2n\right) ^{2}}{%
\left( 1-\frac{\sin j/2n}{j/2n}\right) }\right) ^{k}\left( 1-\frac{\sin j/2n%
}{j/2n}\right) ^{k}A_{j}\left( \cdot ,\widetilde{T_{n}}\right) \right\Vert
_{\varphi ,W}.
\end{eqnarray*}%
Using Marcinkiewicz type multiplier Theorem \ref{mmt} we have%
\begin{eqnarray*}
n^{-2k}\left\Vert T_{n}^{\left( 2k\right) }\right\Vert _{\varphi ,W}
&\lesssim &\left\Vert \sum\limits_{j=1}^{n}\left( 1-\frac{\sin j/2n}{j/2n}%
\right) ^{k}A_{j}\left( \cdot ,T_{n}\right) \right\Vert _{\varphi ,W}+ \\
&&+\left\Vert \sum\limits_{j=1}^{n}\left( 1-\frac{\sin j/2n}{j/2n}\right)
^{k}A_{j}\left( \cdot ,\widetilde{T_{n}}\right) \right\Vert _{\varphi ,W} \\
&=&\left\Vert \sum\limits_{j=1}^{n}\left( 1-\frac{\sin j/2n}{j/2n}\right)
^{k}A_{j}\left( \cdot ,T_{n}\right) \right\Vert _{\varphi ,W}+ \\
&&+\left\Vert \left( \sum\limits_{j=1}^{n}\left( 1-\frac{\sin j/2n}{j/2n}%
\right) ^{k}A_{j}\left( \cdot ,T_{n}\right) \right) ^{\thicksim }\right\Vert
_{\varphi ,W}.
\end{eqnarray*}

In the last step we used the linearity property of conjugate operator. Thus
from boundedness of conjugate (see e.g., (\ref{esen})) operator we get

\begin{eqnarray*}
n^{-2k}\left\Vert T_{n}^{\left( 2k\right) }\right\Vert _{\varphi ,W}
&\lesssim &\left\Vert \sum\limits_{j=1}^{n}\left( 1-\frac{\sin j/2n}{j/2n}%
\right) ^{k}A_{j}\left( \cdot ,T_{n}\right) \right\Vert _{\varphi ,W} \\
&\lesssim &\left\Vert (I-\mathcal{A}_{1/n})^{k}T_{n}\right\Vert _{\varphi
,W}\lesssim \Omega _{k}\left( T_{n},1/n\right) _{\varphi ,W}.
\end{eqnarray*}
\end{proof}

\begin{proof}[\textbf{Proof of Lemma \protect\ref{kubu}}]
Since $0<t<\infty $ there exists some $n\in \mathbb{N}$ so that $(1/n)<t\leq
(2/n)$ holds. Using Lemma \ref{lem2} we have 
\begin{equation}
\Omega _{k}\left( f,t\right) _{\varphi ,W}\leq \Omega _{k}\left(
f-T_{n},t\right) _{\varphi ,W}+\Omega _{k}\left( T_{n},t\right) _{\varphi
,W}\lesssim E_{n}\left( f\right) _{\varphi ,W}+t^{2k}\Vert T_{n}^{\left(
2k\right) }\Vert _{\varphi ,W}.  \label{laz1}
\end{equation}%
On the other hand (\ref{so}) gives $n^{2k}E_{n}\left( f\right) _{\varphi
,W}\lesssim E_{n}\left( f^{\left( 2k\right) }\right) _{\varphi ,W}$. Using
this and Theorem \ref{teo2} we get ($\beta \in \mathbb{R}^{+}$)%
\begin{equation}
E_{n}\left( f\right) _{\varphi ,W}\lesssim \frac{1}{n^{2k}}E_{n}\left(
f^{\left( 2k\right) }\right) _{\varphi ,W}\lesssim \frac{1}{n^{2k}}\Omega
_{\beta }\left( f^{\left( 2k\right) },1/n\right) _{\varphi ,W}\lesssim
t^{2k}\Vert f^{\left( 2k\right) }\Vert _{\varphi ,W}.  \label{laz2}
\end{equation}
\end{proof}

\end{document}